\documentclass[twoside,a4paper,12pt,centertags]{amsart}
\usepackage{amsmath,amssymb,verbatim,vmargin}
\usepackage[bookmarks=true]{hyperref}   
\usepackage{color}
\usepackage{graphicx}

 \newtheorem{thm}{Theorem}[section]
 
 \newtheorem{lem}[thm]{Lemma}
 \newtheorem{prop}[thm]{Proposition}
 \theoremstyle{definition}
 \newtheorem{defn}[thm]{Definition}
 \theoremstyle{remark}
 \newtheorem{rem}[thm]{Remark}

\mathchardef\semic="303B

\newcommand{\barint}{\mbox{$ave \int$}}

\newcommand{\RNum}[1]{\uppercase\expandafter{\romannumeral #1\relax}}

\def\barint_#1{\mathchoice
            {\mathop{\vrule width 6pt
height 3 pt depth -2.5pt
                    \kern -8.8pt
\intop}\nolimits_{#1}}%
            {\mathop{\vrule width 5pt height
3 pt depth -2.6pt
                    \kern -6.5pt
\intop}\nolimits_{#1}}%
            {\mathop{\vrule width 5pt height
3 pt depth -2.6pt
                    \kern -6pt
\intop}\nolimits_{#1}}%
            {\mathop{\vrule width 5pt height
3 pt depth -2.6pt
          \kern -6pt \intop}\nolimits_{#1}}}

\usepackage{color}

\definecolor{gr}{rgb}   {0.,   0.8,   0. }
\definecolor{bl}{rgb}   {0.,   0.5,   1. }
\definecolor{mg}{rgb}   {0.7,  0.,    0.7}
\definecolor{brown}{rgb}{1,0.6,0}
\definecolor{chocolate(traditional)}{rgb}{0.48, 0.25, 0.0}
\definecolor{darkbrown}{rgb}{0.4, 0.26, 0.13}

\makeatletter
\@namedef{subjclassname@2010}{%
  \textup{2010} Mathematics Subject Classification}
\makeatother

\begin{document}

\title[Evolution of time-harmonic wave]
{Evolution of time-harmonic electromagnetic and acoustic waves along waveguides}
\author{Medet Nursultanov}
\address{Department of Mathematics, Chalmers University of Technology and The University of Gothenburg, Sweden}
\email{medet.nursultanov@gmail.com}

\author{Andreas Ros\'en$^*$}
\thanks{$^*$Andreas Ros\'en was formerly named Andreas Axelsson.}
\address{Department of Mathematics, Chalmers University of Technology and The University of Gothenburg, Sweden}
\email{andreas.rosen@chalmers.se}
\maketitle

\begin{abstract}
We study time-harmonic electromagnetic and acoustic waveguides, modeled by an infinite cylinder with a non-smooth cross section. We introduce an infinitesimal generator for the wave evolution along the cylinder, and prove estimates of the functional calculi of these first order non-self adjoint differential operators with non-smooth coefficients. Applying our new functional calculus, we obtain a one-to-one correspondence between polynomially bounded time-harmonic waves and functions in appropriate spectral subspaces.
\end{abstract}


%
%
%

\section{Introduction}

A linear partial differential equation, PDE, or a system of PDEs, is often analyzed by studying the evolution of solutions $u$ with respect to one of the variables, say $t$. In this way the PDE becomes a vector-valued ordinary differential equation, ODE, like
\begin{equation}\label{Int.ode}
  \partial_t u(t,x) + T u(t,x)=0
\end{equation}
in the homogeneous case. We assume here that our PDE is of first order. If it is of a higher order, we first rewrite it as a system of first order equations. Here $T$, an infinitesimal generator, is a first order differential operator acting in the remaining variables $x$ only, for each fixed $t$.

Formally solutions to \eqref{Int.ode} are given by

\begin{equation}\label{Int.sol}
  u(t,x)= (\exp(-tT)u(0,\cdot))(x).
\end{equation}
However, as $T$ is an unbounded operator, we need to be careful in the
definition and analysis of such a solution operator $\exp(-tT)$.
The heuristics are as follows.
For a parabolic equation, say the heat equation, $T$ is the positive Laplace operator and $\exp(-tT)$ is a well defined bounded operator for any $t\ge 0$ and any initial function.
For a hyperbolic equation, say the wave equation as a first order system, $T$ is skew symmetric and $\exp(-tT)$ is unitary and well defined for any $-\infty<t<\infty$ and any initial function.
For an elliptic equation, say the Cauchy--Riemann system, $T$ is symmetric but with spectrum running from $-\infty$ to $+\infty$.
In this case we need to split the function space for initial data as a direct sum of two Hardy subspaces. Then $\exp(-tT)$ is well defined and bounded for $t>0$ when the initial data is in one of the Hardy subspaces, and for $t<0$ when initial data is in the other Hardy subspace.

The aim of the present paper is to study infinitesimal generators $T$ arising as above in the elliptic case. Our motivation comes from the theory for waveguides, and our results yield a powerful mathematical representation of time-harmonic waves propagating along waveguides with general non-smooth materials. The waveguide is modeled by the unbounded region $\mathrm{R}\times \Omega$, where $\Omega$ is a bounded domain in $\mathrm{R}^2$, or more generally in $\mathrm{R}^n$. Note that we study time-harmonic waves. Therefore the PDE is elliptic rather than hyperbolic, and $t$ is not time but rather the spatial variable along the waveguide. For an acoustic waveguide, the PDE is of Helmholtz type, as in Section \ref{Helmholtz subsection}, with coefficients which we allow to vary non-smoothly over the cross section $\Omega$, but they are homogeneous along the waveguide.
For an electromagnetic waveguide, the system of PDEs is Maxwell's equations as we describe in Section \ref{Maxwell subsection}.

We show in Section \ref{pde as ode} that the infinitesimal generators $T$ arising in this way when studying waveguide propagation are of the form
\begin{equation}\label{Int.T}
  T=(D_1+D_0)B,
\end{equation}
where $D_1$ is a self-adjoint first-order differential operator, $D_0$ is a normal bounded multiplication operator and $B$ is a bounded accretive operator depending on the material properties of the cross section of the waveguide. With such variable coefficients, the operator $T$ will not be self-adjoint. Even in the static case $D_0=0$, $T$ is only a bi-sectoral operator (see \cite{PAA}) and $L^2(\Omega)$ bounds of $\exp(-tT)$, and more general functions $f(T)$ of $T$, is a non-trivial matter. However, in the general non-smooth case, this is well understood from the works of Axelsson, Keith and McIntosh \cite{AKM} and Auscher, Axelsson and McIntosh \cite{AAM}.
In the present paper we extend these results to the case $D_0\ne 0$ which
occurs for example in general time-harmonic, but non-static, wave propagation in waveguides.

In Section \ref{SP and FC} we study functional calculi of operators of the form \eqref{Int.T},
which we show have $L^2(\Omega)$ spectra contained in regions
\begin{equation*}
  S_{\omega,\tau}:=\{x+iy\in \mathrm{C}: |y|<|x|\tan \omega+\tau\}.
\end{equation*}
To have a theory for general frequencies of oscillation, encoded by the zero-order term $D_0$, it is essential
to require the cross section $\Omega$ to be bounded, which ensures that the spectrum is discrete. However, the compactness of resolvents and the discreteness of spectrum only holds for $T$ in the range of $D_1+D_0$, which is invariant under $T$.
Building on fundamental quadratic estimates (see \cite{AlbrechtDuongMcIntosh}) for operators $T$ in the static case, we are able to construct and prove $L^2(\Omega)$ estimates of
a generalised Riesz--Dunford functional calculus of $T$. To yield a well defined and bounded operator $f(T)$, the symbol $f(z)$ is required to be uniformly bounded and holomorphic on an open neighbourhood of the spectrum of $T$ except at $\infty$, where it is only required to be bounded and holomorphic on a bi-sector $|y|\le \tan\omega |x|$, $\omega<\pi/2$, in a neighbourhood of $\infty$. Due to the deep quadratic estimates from harmonic analysis used in Proposition \ref{quadratic estimate}, this suffices to bound $f(T)$ at $\infty$.

Another novelty in estimating $f(T)$, due to the non-self adjointness of $T$, is that $\|f(T)\|$ may depend not only on $|f(\lambda)|$, but also on a finite number of derivatives $f^{(k)}(\lambda)$ at a given eigenvalue $\lambda$ of $T$.
In particular, an eigenvalue of $T$ on the imaginary axis with index/algebraic multiplicity greater than $1$, will result in propagating waves $u_t= \exp(-tT)u_0$ which grow polynomially.

Note that since the spectrum is discrete, a symbol like
\begin{equation*}
  f(z)=\begin{cases}
    e^{-tz}, & \mbox{if } \mathrm{Re} z>a, \\
    0, & \mbox{if } \mathrm{Re} z\leq a,
  \end{cases}
\end{equation*}
for $t>0$, is admissible provided no eigenvalue lies on $\mathrm{Re} z=a$,
and will yield an operator bounded on $L^2(\Omega)$.
In this sense the functional calculus that we here construct is more general than that considered by Morris in \cite{Morris}.

In the final Section \ref{Applications}, we apply our  new functional calculus for operators $T$ to show how all polynomially bounded time-harmonic waves in the semi- or bi-infinite waveguide can be represented
like \eqref{Int.sol}, with $u_0$ in appropriate spectral subspace for $T$.

\section{Partial differential equations expressed as vector-valued ordinary differential equation}\label{pde as ode}

In this section we consider Helmholtz and Maxwell's equations and express them as vector-valued ordinary differential equations in terms of operator $DB$, which is introduced later.

Throughout this paper $\Omega=\Omega^+\subset \mathrm{R}^n$ denotes bounded open set, separated from the exterior domain $\Omega^-=\mathrm{R}^n\setminus\Omega$ by weakly Lipschitz interface $\Gamma=\partial\Omega$, defined as follows.

\begin{defn}
The interface $\Gamma$ is weakly Lipschitz if, for all $y\in\Gamma$, there exists a neighbourhood $V_y\ni y$ and a global bilipschitz map $\rho_y:\mathrm{R}^n\rightarrow\mathrm{R}^n$ such that
\begin{equation*}
  \Omega^{\pm}\cap V_y=\rho_y\left(\mathrm{R}^{n}_{\pm}\right)\cap V_y,
\end{equation*}
\begin{equation*}
  \Gamma\cap V_y=\rho_y\left(\mathrm{R}^{n-1}\right)\cap V_y,
\end{equation*}
were $\mathrm{R}^n_+=\mathrm{R}^{n-1}\times(0,+\infty)$ and $\mathrm{R}^n_-=\mathrm{R}^{n-1}\times(-\infty,0)$. In this case $\Omega$ is called a weakly Lipschitz domain.
\end{defn}

\subsection{Helmholtz equation}\label{Helmholtz subsection}
Let $\Omega\subset \mathrm{R}^n$ be a bounded weakly Lipschitz domain and $A\in L_{\infty}\left(\Omega; \mathcal{L}\left(\mathrm{C}^{n+2}\right)\right)$ be $t$-independent and pointwise strictly accretive in the sense that there exist $\alpha>0$ such that
\begin{equation}\label{accretive}
  \text{Re}(A(x)v,v)\geq\alpha\|v\|^2
\end{equation}
for all $x\in \mathrm{R}^n$ and $v\in \mathrm{C}^{n+2}$. For complex number $k\neq0$, we consider an equation
\begin{equation}\label{Helm main eq}
\begin{bmatrix}
\text{div}_{(t,x)}&k\\
\end{bmatrix}
A
\begin{bmatrix}
\nabla_{(t,x)}\\
k
\end{bmatrix}
u=0
\end{equation}
in $\Omega\times\mathrm{R}$ with $u\in H^1_0(\Omega)$ for all $t\in \mathrm{R}$.

Let us set
$$
  H_{\text{div}}(\Omega;\mathrm{C}^{n}):=\{f\in L_2(\Omega;\mathrm{C}^{n}): \; \text{div}f\in L_2(\Omega)\}
$$
and define divergence and gradient operators, respect to an argument $x$, with domains $H_{\text{div}}(\Omega)$ and $H^1_0(\Omega)$ by $\text{div}$ and $\nabla_0$ respectively.

Splitting $\mathrm{C}^{n+2}$ in to $\mathrm{C}$ and $\mathrm{C}^{n+1}$, we decompose the matrix $A(x)$ in the following way
$$
A(x)=\begin{bmatrix}
A_{\perp\perp}(x)&A_{\perp\parallel}(x)\\
A_{\parallel\perp}(x)&A_{\parallel\parallel}(x)\\
\end{bmatrix}.
$$
Then we can write the equation \eqref{Helm main eq} in form
$$
\begin{bmatrix}
\partial_t&
\begin{bmatrix}
\textrm{div}&k
\end{bmatrix}\\
\end{bmatrix}
\begin{bmatrix}
A_{\perp\perp}(x)&A_{\perp\parallel}(x)\\
A_{\parallel\perp}(x)&A_{\parallel\parallel}(x)\\
\end{bmatrix}
\begin{bmatrix}
\partial_{t}u\\
\begin{bmatrix}
\nabla_0 u\\
ku
\end{bmatrix}
\end{bmatrix}=0.
$$
Hence
\begin{equation}\label{1e2}
\begin{bmatrix}
\partial_t&
\begin{bmatrix}
\textrm{div}&k
\end{bmatrix}\\
\end{bmatrix}
\begin{bmatrix}
A_{\perp\perp}\partial_{t}u+A_{\perp\parallel}
\begin{bmatrix}\nabla_0 u\\ku\end{bmatrix}\\
A_{\parallel\perp}\partial_{t}u+A_{\parallel\parallel}\begin{bmatrix}
\nabla_0 u\\
ku
\end{bmatrix}
\end{bmatrix}=0.
\end{equation}
Next define $f$ as
\begin{equation}\label{def of f}
f=
\begin{bmatrix}
f_{\perp}\\
f_{\parallel}
\end{bmatrix}:=
\begin{bmatrix}
A_{\perp\perp}\partial_{t}u+A_{\perp\parallel}
\begin{bmatrix}\nabla_0 u\\ku\end{bmatrix}\\
\begin{bmatrix}
\nabla_0 u\\
ku
\end{bmatrix}
\end{bmatrix}.
\end{equation}
Since $A$ is pointwise strictly accretive, all diagonal blocks are pointwise strictly accretive, and consequently invertible. In particular, $A_{\perp\perp}$ is invertible. Hence, due to \eqref{def of f}, we obtain $\partial_tu=A_{\perp\perp}^{-1}(f_{\perp}-A_{\perp\parallel}f_{\parallel})$. Therefore we can write equation \eqref{1e2} in terms of $f$
$$
\begin{bmatrix}
\partial_t&
\begin{bmatrix}
\textrm{div}&k
\end{bmatrix}\\
\end{bmatrix}
\begin{bmatrix}
A_{\perp\perp}A_{\perp\perp}^{-1}(f_{\perp}-A_{\perp\parallel}f_{\parallel})+A_{\perp\parallel}
f_{\parallel}\\
A_{\parallel\perp}A_{\perp\perp}^{-1}(f_{\perp}-A_{\perp\parallel}f_{\parallel})+A_{\parallel\parallel}f_{\parallel}
\end{bmatrix}=0,
$$
hence
\begin{equation}\label{1e3}
\begin{bmatrix}
\partial_t&
\begin{bmatrix}
\textrm{div}&k
\end{bmatrix}\\
\end{bmatrix}
\begin{bmatrix}
f_{\perp}\\
A_{\parallel\perp}A_{\perp\perp}^{-1}(f_{\perp}-A_{\perp\parallel}f_{\parallel})+A_{\parallel\parallel}f_{\parallel}
\end{bmatrix}=0.
\end{equation}
On the other hand, from definition of $f_{\parallel}$, we obtain
$$
\partial_tf_{\parallel}=
\begin{bmatrix}
\nabla_0 \partial_t u\\
k\partial_tu
\end{bmatrix}=
\begin{bmatrix}
\nabla_0 \\
k
\end{bmatrix}
(A_{\perp\perp}^{-1}(f_{\perp}-A_{\perp\parallel}f_{\parallel})),
$$
which, together with \eqref{1e3}, give us the system of equations
$$
\begin{cases}
\partial_tf_{\perp}+
\begin{bmatrix}
\textrm{div}&k
\end{bmatrix}
(A_{\parallel\perp}A_{\perp\perp}^{-1}(f_{\perp}-A_{\perp\parallel}f_{\parallel})+A_{\parallel\parallel}f_{\parallel})=0 \\
\partial_tf_{\parallel}-
\begin{bmatrix}
\nabla_0\\
k
\end{bmatrix}
A_{\perp\perp}^{-1}(f_{\perp}-A_{\perp\parallel}f_{\parallel})=0.
\end{cases}
$$
In vector notation, we equivalently have
$$
\partial_t
\begin{bmatrix}
f_{\perp}\\
f_{\parallel}
\end{bmatrix}
+
\begin{bmatrix}
0&
\begin{bmatrix}
\textrm{div}&k
\end{bmatrix}\\
-\begin{bmatrix}
\nabla_0\\
k
\end{bmatrix}
&0
\end{bmatrix}
\begin{bmatrix}
A_{\perp\perp}^{-1}&-A_{\perp\perp}^{-1}A_{\perp\parallel}\\
A_{\parallel\perp}A_{\perp\perp}^{-1}&A_{\parallel\parallel}-A_{\parallel\perp}A_{\perp\perp}^{-1}A_{\perp\parallel}
\end{bmatrix}
\begin{bmatrix}
f_{\perp}\\
f_{\parallel}
\end{bmatrix}=0.
$$
Define
$$
B:=
\begin{bmatrix}
A_{\perp\perp}^{-1}&-A_{\perp\perp}^{-1}A_{\perp\parallel}\\
A_{\parallel\perp}A_{\perp\perp}^{-1}&A_{\parallel\parallel}-A_{\parallel\perp}A_{\perp\perp}^{-1}A_{\perp\parallel}
\end{bmatrix}
$$
and
$$
D:=
\begin{bmatrix}
0&
\begin{bmatrix}
\textrm{div}&k
\end{bmatrix}\\
-\begin{bmatrix}
\nabla_{0}\\
k
\end{bmatrix}&0
\end{bmatrix}
$$
with domain
$$
\mathbf{D}(D)=\{f=(f_{1},f_2,f_3)\in L_{2}(\Omega; \mathrm{C}^{2+n}): \; f_{1}\in H_0^1(\Omega),
$$
$$
f_{2}\in H_{\text{div}}(\Omega;\mathrm{C}^n), \; f_{3}\in L_{2}(\Omega)\},
$$
so that the equation becomes
\begin{equation}\label{first order equation}
\partial_tf+DBf=0,
\end{equation}
together with the constraint $f\in\mathbf{R}(D)$ for each fixed $t\in \mathrm{R}$.

Since $A$ is a pointwise strictly accretive operator, in [\cite{AAM}, Proposition 3.2] it was noted that $B$ is a strictly accretive multiplication operator just like $A$.

By the above arguments, the equation \eqref{Helm main eq} for $u$ implies that $f$, defined above, solves the equation \eqref{first order equation}. Moreover, the converse is also true, i.e. the following proposition holds.

\begin{prop}\label{Helm second implies first}
If $(f,\nabla_0 g, kg)\in\mathbf{R}(D)$ solves equation \eqref{first order equation}, then $g$ solves equation \eqref{Helm main eq}.
\end{prop}

\begin{proof}
Let $(f,\nabla_0 g, kg)\in\mathbf{R}(D)$ be a solution of equation \eqref{first order equation}, then
\begin{equation}\label{system first impies second}
\begin{cases}
\partial_tf+
\begin{bmatrix}
\textrm{div}&k
\end{bmatrix}
\left(A_{\parallel\perp}A_{\perp\perp}^{-1}\left(f-A_{\perp\parallel}\begin{bmatrix}\nabla_0 g\\kg \end{bmatrix}\right)+A_{\parallel\parallel}\begin{bmatrix}\nabla_0 g\\kg \end{bmatrix}\right)=0, \\
\partial_t\begin{bmatrix}\nabla_0 g\\kg \end{bmatrix}-
\begin{bmatrix}
\nabla_0\\
k
\end{bmatrix}
A_{\perp\perp}^{-1}\left(f-A_{\perp\parallel}\begin{bmatrix}\nabla_0 g\\kg \end{bmatrix}\right)=0.
\end{cases}
\end{equation}
The first equation of \eqref{system first impies second} can be written in form
\begin{equation}\label{section2.system.first equation}
\begin{bmatrix}
\partial_t&
\begin{bmatrix}
\textrm{div}&k
\end{bmatrix}\\
\end{bmatrix}
\begin{bmatrix}
f\\
A_{\parallel\perp}A_{\perp\perp}^{-1}(f-A_{\perp\parallel}\begin{bmatrix}\nabla_0 g\\kg \end{bmatrix})+A_{\parallel\parallel}\begin{bmatrix}\nabla_0 g\\kg \end{bmatrix}
\end{bmatrix}=0.
\end{equation}
From the second equation of the system \eqref{system first impies second}, we see
\begin{equation}\label{sec2.1}
\partial_t g=
A_{\perp\perp}^{-1}\left(f-A_{\perp\parallel}\begin{bmatrix}\nabla_0 g\\kg \end{bmatrix}\right),
\end{equation}
thus
\begin{equation}\label{sec2.2}
f=A_{\perp\perp}\partial_t g+A_{\perp\parallel}\begin{bmatrix}\nabla_0 g\\kg \end{bmatrix}.
\end{equation}
Setting \eqref{sec2.1} and \eqref{sec2.2} in to the formula \eqref{section2.system.first equation}, we get
$$
\begin{bmatrix}
\partial_t&
\begin{bmatrix}
\textrm{div}_{\parallel}&k
\end{bmatrix}\\
\end{bmatrix}
\begin{bmatrix}
A_{\perp\perp}\partial_t g+A_{\perp\parallel}\begin{bmatrix}\nabla_0 g\\kg \end{bmatrix}\\
A_{\parallel\perp}\partial_t g+A_{\parallel\parallel}\begin{bmatrix}\nabla_0 g\\kg \end{bmatrix}
\end{bmatrix}=0.
$$
This shows that $g$ solves equation \eqref{Helm main eq}.
\end{proof}

Let us define operators

\begin{equation*}
  D_1:=\begin{bmatrix}
         0 & \text{div} & 0  \\
         -\nabla_0 & 0 & 0  \\
         0 & 0 & 0
       \end{bmatrix},
       \qquad
       D_0:=\begin{bmatrix}
              0 & 0 & k \\
              0 & 0 & 0 \\
              -k & 0 & 0
            \end{bmatrix}
\end{equation*}
with domains $\mathbf{D}(D)$ and $L_2(\Omega; \mathrm{C}^{n+2})$. Then
\begin{equation*}
  D=D_1+D_0.
\end{equation*}

\begin{rem}
Note that $D_1$ is a self-adjoint operator, see [\cite{KZ}, theorem 6.2], and $D_0$ is a bounded operator. Therefore $D$ is a closed operator and
\begin{equation*}
  D^*=D_1^*+D_0^*=\begin{bmatrix}
         0 & \text{div} & -\overline{k} \\
         -\nabla_0 & 0 & 0 \\
         \overline{k} & 0 & 0
       \end{bmatrix}.
\end{equation*}
\end{rem}


\subsection{Maxwell's equation}\label{Maxwell subsection}
Let $\Omega\subset \mathrm{R}^2$ be a bounded weakly Lipschitz domain. By Rademachers's theorem the surface $\partial\Omega$ has a tangent plane and an outward pointing unit normal $n(x)$ at almost every $x\in \partial\Omega$. We introduce the Sobolev spaces
\begin{equation*}
  H_{\textrm{div}}(\Omega;\mathrm{C}^2):=\{f\in L_2(\Omega;\mathrm{C}^2): \; \textrm{div}f\in L_2(\Omega)\},
\end{equation*}
\begin{equation*}
  H_{\textrm{curl}}(\Omega;\mathrm{C}^2):=\{f\in L_2(\Omega;\mathrm{C}^2): \; \textrm{curl}f\in L_2(\Omega)\}
\end{equation*}
and
\begin{equation*}
  H^0_{\textrm{div}}(\Omega;\mathrm{C}^2):=\{f\in H_{\textrm{div}}(\Omega;\mathrm{C}^2): \; \textrm{div}(\tilde{f})\in L_2(\mathrm{R}^2)\},
\end{equation*}
\begin{equation*}
  H^0_{\textrm{curl}}(\Omega;\mathrm{C}^2):=\{f\in H_{\textrm{curl}}(\Omega;\mathrm{C}^2): \; \textrm{curl}(\tilde{f})\in L_2(\mathrm{R}^2)\}.
\end{equation*}
where $\tilde{f}$ denotes the zero-extension of $f$ to $\mathrm{R}^2$.

The last two spaces have the following geometric meaning. Assume that $f\in H^0_{\textrm{div}}(\Omega;\mathrm{C}^2)$, then there exists a sequence $\{\psi_k\}_{k=1}^{\infty}\subset C_0^{\infty}(\mathrm{R}^2;\mathrm{C}^2)$ such that $\psi_k\rightarrow f$ and $\textrm{div}\psi_k\rightarrow \textrm{div}f$, see [\cite{AndreasRosen}, Definition 8.14, Lemma 8.18]. Hence for $\phi\in C_0^{\infty}(\mathrm{R}^2)$, we obtain
\begin{equation*}
  \int_{\Omega}(\textrm{div}f,\phi)-\int_{\Omega}(f,-\nabla\phi)=
  \lim_{k\rightarrow\infty}\left(\int_{\Omega}(\textrm{div}\psi_k,\phi)-\int_{\Omega}(\psi_k,-\nabla\phi)\right)=0.
\end{equation*}
Hence the Stokes' theorem implies formally
\begin{equation*}
  \int_{\partial\Omega}(f\cdot n,\phi)=\int_{\Omega}(\textrm{div}f,\phi)-\int_{\Omega}(f,-\nabla\phi)=0.
\end{equation*}
Therefore we interpret $f\in H^0_{\textrm{div}}(\Omega;\mathrm{C}^2)$ as saying, beside $\text{div}f\in L_2(\Omega)$, that $f$ is tangential on the boundary in a weak sense. Similarly, the condition $f\in H^0_{\textrm{curl}}(\Omega;\mathrm{C}^2)$ means $\textrm{curl}f\in L_2(\Omega)$ and that $f$ is normal on the boundary in a weak sense.

By $\nabla$, $\nabla_0$, $\textrm{div}$ and $\textrm{div}_0$, we define gradient and divergence operators on $H^1(\Omega)$, $H^1_0(\Omega)$, $H_{\textrm{div}}(\Omega;\mathrm{C}^2)$ and $H^0_{\textrm{div}}(\Omega;\mathrm{C}^2)$ respectively.

\begin{rem}\label{WHdiv equals Hcurl}
For bounded weakly Lipschitz domain $\Omega\subset \mathrm{R}^2$ and function $f\in H_{\textrm{div}}(\Omega;\mathrm{C}^2)$, we see
\begin{equation*}
  \textrm{curl}Jf=\textrm{div} f,
  \qquad
  f\cdot n=Jf\times n
\end{equation*}
where
\begin{equation*}
  J=\begin{bmatrix}
      0 & -1 \\
      1 & 0
    \end{bmatrix}.
\end{equation*}
This gives
\begin{equation*}
  JH_{\text{div}}(\Omega;\mathrm{C}^2)=H_{\text{curl}}(\Omega;\mathrm{C}^2),
  \qquad
  JH_{\text{div}}^{0}(\Omega;\mathrm{C}^2)=H_{\text{curl}}^{0}(\Omega;\mathrm{C}^2).
\end{equation*}
\end{rem}

Let $\mu(x),\varepsilon(x)\in L_{\infty}\left(\mathrm{R}^2;\mathcal{L}\left(\mathrm{C}^3\right)\right)$ be pointwise strictly accretive matrices, see \eqref{accretive}. For complex number $\omega\neq 0$, we consider the Maxwell's system of equations
\begin{equation}\label{Maxwell's equation}
  \begin{cases}
    \text{div}_{(t,x)}\mu H=0, \\
    i \omega\mu H+\text{curl}_{(t,x)}E=0,\\
    i \omega\varepsilon E-\text{curl}_{(t,x)}H=0,\\
    \text{div}_{(t,x)}\varepsilon E=0\\
  \end{cases}
\end{equation}
in $\mathrm{R}\times\Omega$ with
\begin{equation*}
  \mu H\in L_2(\Omega)\times H^0_{\textrm{div}}(\Omega;\mathrm{C}^2),
\end{equation*}
\begin{equation*}
  E\in L_2(\Omega)\times H^0_{\textrm{curl}}(\Omega;\mathrm{C}^2)
\end{equation*}
for any fixed $t\in \mathrm{R}$.

According to the splitting $\mathrm{C}^3$ into $\mathrm{C}$ and $\mathrm{C}^2$, we write
\begin{equation*}
  H=\begin{bmatrix}
      H_{\perp} \\
      H_{\parallel}
    \end{bmatrix},
  \quad
  E=\begin{bmatrix}
      E_{\perp} \\
      E_{\parallel}
    \end{bmatrix},
\end{equation*}

\begin{equation*}
  \mu=\begin{bmatrix}
         \mu_{\perp\perp} & \mu_{\perp\parallel} \\
         \mu_{\parallel\perp} & \mu_{\parallel\parallel}
       \end{bmatrix},
       \quad
  \varepsilon=\begin{bmatrix}
                 \varepsilon_{\perp\perp} & \varepsilon_{\perp\parallel} \\
                 \varepsilon_{\parallel\perp} & \varepsilon_{\parallel\parallel}
       \end{bmatrix}
\end{equation*}
and define auxiliary matrices
\begin{equation*}
  \overline{\mu}:=
  \begin{bmatrix}
  \mu_{\perp\perp} & \mu_{\perp\parallel} \\
  0 & I
  \end{bmatrix},
  \quad
  \underline{\mu}:=
       \begin{bmatrix}
         1 & 0 \\
         \mu_{\parallel\perp} & \mu_{\parallel\parallel}
       \end{bmatrix},
\end{equation*}

\begin{equation*}
  \overline{\varepsilon}:=
  \begin{bmatrix}
  \varepsilon_{\perp\perp} & \varepsilon_{\perp\parallel} \\
  0 & I
  \end{bmatrix},
  \quad
  \underline{\varepsilon}:=
  \begin{bmatrix}
  1 & 0 \\
  \varepsilon_{\parallel\perp} & \varepsilon_{\parallel\parallel}
  \end{bmatrix},
\end{equation*}
\begin{equation*}
A=\begin{bmatrix}
      \mu & 0 \\
      0 & \varepsilon
    \end{bmatrix},
    \quad
\overline{A}:=\begin{bmatrix}
                    \overline{\mu} & 0 \\
                    0 & \overline{\varepsilon}
                  \end{bmatrix},
                  \qquad
\underline{A}:=\begin{bmatrix}
                    \underline{\mu} & 0 \\
                    0 & \underline{\varepsilon}
                    \end{bmatrix}.
\end{equation*}

Since $\mu$, $\varepsilon$ are pointwise strictly accretive, we conclude that $\mu_{\perp\perp}$, $\varepsilon_{\perp\perp}$ are pointwise strictly accretive, and consequently $\overline{\mu}$, $\overline{\varepsilon}$ and $\overline{A}$ are invertible.

Let $I_{\perp}=\{I_{\perp}^{i,j}\}_{i,j=1}^6$ be a $6$ by $6$ matrix such that $I_{\perp}^{1,1}=I_{\perp}^{4,4}=1$ and other elements are zeros. We set $I_{\parallel}=I-I_{\perp}$. From the first and forth equations of \eqref{Maxwell's equation}, we get
\begin{equation}\label{Max eq1}
  \partial_{t}I_{\perp}AG+\begin{bmatrix}
                               0 & \text{div}_0 & 0 & 0 \\
                               -\nabla & 0 & 0 & 0 \\
                               0 & 0 & 0 & \text{div} \\
                               0 & 0 & -\nabla_0 & 0
                             \end{bmatrix}I_{\parallel}AG=0,
  \quad
  \text{where}
  \quad
  G:=\begin{bmatrix}
       H \\
       G
     \end{bmatrix}.
\end{equation}
From the second and third equations of \eqref{Maxwell's equation}, we obtain
\begin{equation}\label{Max eq2}
  \partial_{t}I_{\parallel}G+\begin{bmatrix}
                               0 & \text{div}_0 & 0 & 0 \\
                               -\nabla & 0 & 0 & 0 \\
                               0 & 0 & 0 & \text{div} \\
                               0 & 0 & -\nabla_0 & 0
                             \end{bmatrix}I_{\perp}G-\begin{bmatrix}
                                                        0 & 0 & 0 & 0 \\
                                                        0 & 0 & 0 & i\omega J \\
                                                        0 & 0 & 0 & 0 \\
                                                        0 & -i\omega J & 0 & 0
                                                      \end{bmatrix}I_{\parallel}AG=0.
\end{equation}
Since $I_{\perp}AG=I_{\perp}\overline{A}G$, $I_{\parallel}AG=I_{\parallel}\underline{A}G$, $G_{\parallel}=I_{\parallel}\overline{A}G$ and $I_{\perp}G=I_{\perp}\underline{A}G$, we can combine equations \eqref{Max eq1} and \eqref{Max eq2} in the following way
\begin{equation}\label{Max eq3}
  \partial_{t}\overline{A}G+\begin{bmatrix}
                               0 & \text{div}_0 & 0 & 0 \\
                               -\nabla & 0 & 0 & i\omega J \\
                               0 & 0 & 0 & \text{div} \\
                               0 & -i\omega J & -\nabla_0 & 0
                             \end{bmatrix}\underline{A}G=0.
\end{equation}
Define
\begin{equation*}
  D:=
  \begin{bmatrix}
    0 & \text{div}_0 & 0 & 0 \\
    -\nabla & 0 & 0 & i\omega J \\
    0 & 0 & 0 & \text{div} \\
    0 & -i\omega J & -\nabla_0 & 0
  \end{bmatrix}
\end{equation*}
with domain
$$
\mathbf{D}(D)=\{f=(f_1,f_2,f_3,f_4)\in L_2(\Omega): \; f_1\in H^1(\Omega), \; f_2\in H^{0}_{\text{div}}(\Omega;\mathrm{C}^2),
$$
$$
f_3\in H^1_0(\Omega), \; f_4\in H_{\text{div}}(\Omega;\mathrm{C}^2) \}.
$$

Let $B:=\underline{A}\overline{A}^{-1}$, $F:=\overline{A}G$, so that equation \eqref{Max eq3} becomes
\begin{equation}\label{Max main eq}
  \partial_{t}F+DBF=0
\end{equation}
together with constraint $F\in\mathbf{R}(D)$ for each fixed $t\in \mathrm{R}$.

To see that the system of equations \eqref{Maxwell's equation} and the equation of \eqref{Max main eq} are equivalent, let us prove analogue of Proposition \ref{Helm second implies first}.

\begin{prop}\label{Max second implies first}
Let $f(t,x)$ and $g(t,x)$ be three dimensional vector functions such that $(f,g)$ solves equation \eqref{Max main eq} and $(f,g)\in \mathbf{R}(D)\cap\mathbf{D}(DB)$ for each fixed $t\in\mathrm{R}$, then vector functions
\begin{equation*}
  H=\overline{\mu}^{-1}f,
  \qquad
  E=\overline{\varepsilon}^{-1}g
\end{equation*}
solve the system of equations \eqref{Maxwell's equation} and for any fixed $t\in \mathrm{R}$,
\begin{equation*}
  \mu H\in L_2(\Omega)\times H^0_{\mathrm{div}}(\Omega;\mathrm{C}^2),
\end{equation*}
\begin{equation*}
  E\in L_2(\Omega)\times H^0_{\mathrm{curl}}(\Omega;\mathrm{C}^2).
\end{equation*}
\end{prop}

\begin{proof}
  Splitting $\mathrm{C}^3$ into $\mathrm{C}$ and $\mathrm{C}^2$, we write
  \begin{equation*}
    f=\begin{bmatrix}
       f_{\perp} \\
       f_{\parallel}
     \end{bmatrix},
     \quad
    g=\begin{bmatrix}
       g_{\perp} \\
       g_{\parallel}
     \end{bmatrix},
     \quad
   H =\begin{bmatrix}
       H_{\perp} \\
       H_{\parallel}
     \end{bmatrix},
     \quad
     E=\begin{bmatrix}
       E_{\perp} \\
       E_{\parallel}
     \end{bmatrix}.
  \end{equation*}
  Since $(f,g)$ is a solution for \eqref{Max main eq}, we see
  \begin{equation*}
    \partial_t\overline{A}\begin{bmatrix}
                          H \\
                          E
                        \end{bmatrix}
    +
    DB\overline{A}\begin{bmatrix}
                          H \\
                          E
                        \end{bmatrix}
    =
    \partial_t\overline{A}\begin{bmatrix}
                          H \\
                          E
                        \end{bmatrix}
    +
    D\underline{A}\begin{bmatrix}
                          H \\
                          E
                        \end{bmatrix}=0.
  \end{equation*}
  Thus
  \begin{equation}\label{Max second implies first eq1}
    \begin{cases}
      \partial_t\left(\mu_{\perp\perp}H_{\perp}+\mu_{\perp\parallel}H_{\parallel}\right)+
      \mathrm{div}_0\left(\mu_{\parallel\perp}H_{\perp}+\mu_{\parallel\parallel}H_{\parallel}\right)=0,\\

      \partial_tH_{\parallel}-\nabla H_{\perp}+
      i\omega J\left(\varepsilon_{\parallel\perp}E_{\perp}+\varepsilon_{\parallel\parallel}E_{\parallel}\right)=0,\\

      \partial_t\left(\varepsilon_{\perp\perp}E_{\perp}+\varepsilon_{\perp\parallel}E_{\parallel}\right)+
      \mathrm{div}\left(\varepsilon_{\parallel\perp}E_{\perp}+\varepsilon_{\parallel\parallel}E_{\parallel}\right)=0,\\

      \partial_tE_{\parallel}-\nabla_0 E_{\perp}-
      i\omega J\left(\mu_{\parallel\perp}H_{\perp}+\mu_{\parallel\parallel}H_{\parallel}\right)=0.
    \end{cases}
  \end{equation}
  By assumption, $(f,g)\in \mathbf{R}(D)$ for fixed $t\in \mathrm{R}$, and hance Proposition \ref{Max H equals RD} implies
  \begin{equation*}
    \begin{cases}
      \mathrm{curl}f_{\parallel}-i\omega g_{\perp}=0,\\
      \mathrm{curl}g_{\parallel}+i\omega f_{\perp}=0.
    \end{cases}
  \end{equation*}
  Therefore, in terms of $H$ and $E$, we can write
  \begin{equation}\label{Max second implies first eq2}
    \begin{cases}
      \mathrm{curl}H_{\parallel}-i\omega \left(\varepsilon_{\perp\perp}E_{\perp}+\varepsilon_{\perp\parallel}E_{\parallel}\right)=0,\\
      \mathrm{curl}E_{\parallel}+i\omega \left(\mu_{\perp\perp}H_{\perp}+\mu_{\perp\parallel}H_{\parallel}\right)=0.
    \end{cases}
  \end{equation}
  Combining \eqref{Max second implies first eq1} and \eqref{Max second implies first eq2}, we conclude that $H$, $E$ solve the system of equations \eqref{Maxwell's equation}.

Since $\overline{\mu}H=f$ and $f\in \mathbf{D}(DB)$ for each fixed $t\in \mathrm{R}$, it follows that
$$
\underline{\mu}H\in H^1(\Omega)\times H_{\mathrm{div}}^0(\Omega;\mathrm{C}^2),
$$
hence that
$$
\mu H\in L_2(\Omega)\times H_{\mathrm{div}}^0(\Omega;\mathrm{C}^2).
$$
Proposition \ref{Max H equals RD} and relation between $E$ and $g$ lead to $E_{\parallel}\in H_{\mathrm{curl}}^0(\Omega;\mathrm{C}^2)$, so that for any fixed $t\in \mathrm{R}$,
$$
E\in L_2(\Omega)\times H_{\mathrm{curl}}^0(\Omega;\mathrm{C}^2).
$$
\end{proof}

Let us define operators
\begin{equation*}\label{Max D0 D1}
  D_1:=\begin{bmatrix}
         0 & \text{div}_0 & 0 & 0 \\
         -\nabla & 0 & 0 & 0 \\
         0 & 0 & 0 & \text{div} \\
         0 & 0 & -\nabla_0 & 0
       \end{bmatrix},
       \qquad
       D_0:=\begin{bmatrix}
              0 & 0 & 0 & 0 \\
              0 & 0 & 0 & i\omega J \\
              0 & 0 & 0 & 0 \\
              0 & -i \omega J & 0 & 0
            \end{bmatrix}
\end{equation*}
with domains $\mathbf{D}(D)$ and $L_2(\Omega; \mathrm{C}^{n+2})$. Then
\begin{equation*}
  D=D_1+D_0.
\end{equation*}

\begin{rem}
Note that $D_1$ is a self-adjoint operator, see [\cite{KZ}, theorem 6.2], and $D_0$ is a bounded operator. Therefore $D$ is a closed operator and
\end{rem}
\begin{equation*}
  D^*=D_1^*+D_0^*=\begin{bmatrix}
         0 & \text{div}_0 & 0 & 0 \\
         -\nabla & 0 & 0 & \overline{i\omega} J\\
         0 & 0 & 0 & \text{div} \\
         0 & -\overline{i\omega} J & -\nabla_0 & 0
       \end{bmatrix}.
\end{equation*}

\subsection{Properties of D}
Here we prove that the operators defined in Sections \ref{Helmholtz subsection} and \ref{Maxwell subsection} have closed range and compact resolvent. We will use symbols $\sigma(\cdot)$ and $\rho(\cdot)$ to denote spectrum and resolvent set of an operator.

Let us start by considering operator $D$ defined in Sections \ref{Helmholtz subsection}. First, we prove that the range $\mathbf{R}(D)$ is closed.

\begin{prop}\label{Helm closed range}
Let $\Omega\subset \mathrm{R}^n$ be a bounded weakly Lipschitz domain and $D$ be the operator defined in Section \ref{Helmholtz subsection}. Then the range $\mathbf{R}(D)$ is a closed subspace of $L_2\left(\Omega; \mathrm{C}^{n+2}\right)$.
\end{prop}

\begin{proof}
Let $h\in\mathbf{D}(D^*)$ and $D^*h\in\mathbf{D}(D)$. Then
\begin{equation}\label{Helm.closed range}
\|DD^*h\|^2=
\left\|\begin{matrix}
-\textrm{div}\nabla_0 h_1+|k|^2h_1\\
-\nabla_0(\textrm{div}h_2-\overline{k}h_3)\\
-k(\textrm{div}h_2-\overline{k}h_3)
\end{matrix}\right\|
\end{equation}
$$
=\|\textrm{div}\nabla_0 h_1\|^2+|k|^4\|h_1\|^2+2|k|^2\|\nabla_0 h_1\|^2
$$
$$
+\|\nabla_0(\textrm{div}h_2-\overline{k}h_3)\|^2+\|k(\textrm{div}h_2-\overline{k}h_3)\|^2
$$
$$
\geq |k|^2\|\overline{k}h_1\|^2+|k|^2\|\nabla_0 h_1\|^2+|k|^2\|\textrm{div}h_2-\overline{k}h_3\|^2=|k|^2\|D^*h\|^2.
$$
Since $k\neq0$, the rage $\mathbf{R}(D)$ is closed.
\end{proof}

To prove Proposition \ref{Helm closed range} we use that $k\neq0$, however, by applying Poincar\'e inequality, one can prove it also for $k=0$.

Next, we find the exact expression for the range $\mathbf{R}(D)$.

\begin{prop}
  Let $\Omega\subset \mathrm{R}^n$ be a bounded weakly Lipschitz domain and $D$ be the operator defined in Section \ref{Helmholtz subsection}, then $\mathbf{R}(D)=\mathcal{H}$, where
  \begin{equation*}
    \mathcal{H}:=\left\{f=(f_1,f_2,f_3)\in L_{2}(\Omega; \mathrm{C}^{2+n}): \; f_{3}\in H_0^1(\Omega), \; f_2=\frac{1}{k}\nabla_0 f_{3}\right\}.
  \end{equation*}
\end{prop}
\begin{proof}
  By definition of operator $D$, we obtain that $\mathbf{R}(D)\subset\mathcal{H}$. Conversely, assume that $f=(f_1,f_2,f_3)\in \mathcal{H}$. Since
  \begin{equation*}
    L_2(\Omega)=\mathbf{N}(\nabla_0)\oplus \overline{\mathbf{R}(\textrm{div})},
  \end{equation*}
  there exist some function $h\in\mathbf{N}(\nabla_0)$ and a sequence $\{g^l\}_{l=1}^{\infty}\subset H_{\textrm{div}}(\Omega;\mathrm{C}^n)$ such that $h+\textrm{div}g^l\rightarrow f_1$ in $L_2$ norm. Therefore
  \begin{equation*}
    D\begin{bmatrix}
       -\frac{1}{k}f_3 \\
       g^l \\
       h
     \end{bmatrix}\rightarrow \begin{bmatrix}
                           f_1 \\
                           f_2 \\
                           f_3
                         \end{bmatrix}.
  \end{equation*}
  This, by Proposition \ref{Helm closed range}, implies that $f\in \mathbf{R}(D)$.
\end{proof}

Finally, we prove that the resolvent operators are compact, which implies that the spectrum $\sigma(\left.D\right|_{\mathbf{R}(D)})$ contains only the eigenvalues of $\left.D\right|_{\mathbf{R}(D)}$ and each eigenvalues has a finite geometric multiplicity.

\begin{prop}\label{Helm.compact resolvent}
Let $\Omega\subset \mathrm{R}^n$ be a bounded weakly Lipschitz domain and $D$ be the operator defined in Section \ref{Helmholtz subsection}. Assume $\lambda\in\rho(\left.D\right|_{\mathbf{R}(D)})$, then
$$
(\lambda-\left.D\right|_{\mathbf{R}(D)})^{-1}:\mathbf{R}(D)\rightarrow \mathbf{R}(D)
$$
is a compact operator.
\end{prop}

\begin{proof}
Since $\mathbf{R}(D)\subset L_2(\Omega;\mathrm{C}^n)$ is closed and the operator $D(\lambda-D)^{-1}$ is closed and defined on whole $L_2(\Omega;\mathrm{C}^n)$, we see that
\begin{equation*}
  \left.D\right|_{\mathbf{R}(D)}(\lambda-\left.D\right|_{\mathbf{R}(D)})^{-1}:\mathbf{R}(D)\rightarrow \mathbf{R}(D)
\end{equation*}
is a bounded operator. Therefore, it is suffices to show that the embedding
$$
\left(\mathbf{D}(D)\cap \mathbf{R}(D),\|\cdot\|_{\mathbf{D}(D)\cap \mathbf{R}(D)}\right) \hookrightarrow\left(\mathbf{R}(D),\|\cdot\|_{L_2}\right)
$$
is compact, where
$$
\|f\|_{\mathbf{D}(D)\cap \mathbf{R}(D)}=\|Df\|+\|f\|.
$$

Let $\{(f^l,\nabla g^l,kg^l)\}_{l=1}^{+\infty}$ be a sequence in $\left(\mathbf{D}(D)\cap \mathbf{R}(D),\|\cdot\|_{\mathbf{D}(D)\cap \mathbf{R}(D)}\right)$ such that
\begin{equation}\label{Helm.compact resolvent1}
  \left\|\begin{matrix}
      f^l \\
      \nabla_0 g^l \\
      kg^l
    \end{matrix}\right\|+
    \left\|D\begin{bmatrix}
      f^l \\
      \nabla_0 g^l \\
      kg^l
    \end{bmatrix}\right\|<C
\end{equation}
for some $C>0$. In particular, we get
\begin{equation*}
  \|f\|+\|\nabla_0 f\|\leq C.
\end{equation*}
Therefore, the sequence $\{f^l\}_{l=1}^{\infty}$ is bounded in $H^1(\Omega)$. Since $\Omega\subset\mathrm{R}^n$ is bounded, the Sobolev Embedding Theorem gives that $H^1(\Omega)\hookrightarrow L_2(\Omega)$ is compact. Hence, the sequence $\{f^l\}_{l=1}^{\infty}$ contains a Cauchy subsequence in $L_2(\Omega)$. The same conclusion can be drawn for $\{g^l\}_{l=1}^{\infty}$.

From estimate \eqref{Helm.compact resolvent1}, we obtain
\begin{equation*}
  \|\mathrm{div}\nabla_0 g^l+k^2g^l\|+\|g^l\|\leq C,
\end{equation*}
hence that $\|\mathrm{div}\nabla_0 g^l\|\leq C$. Next, since $\{g^l\}_{l=1}^{\infty}\subset H^1_0(\Omega)$ and $\mathrm{curl}\nabla_0 g^l=0$,
we see that $\{\nabla_0g^l\}_{l=1}^{\infty}$ is a bounded sequence in $H^0_{\mathrm{curl}}(\Omega;\mathrm{C}^n)\cap H_{\mathrm{div}}(\Omega;\mathrm{C}^n)$. Consequently, the sequence $\{\nabla_0g^l\}_{l=1}^{\infty}$ contains a Cauchy subsequence in $L_2(\Omega; \mathrm{C}^n)$, because the embedding
\begin{equation*}
  H^0_{\mathrm{curl}}(\Omega;\mathrm{C}^n)\cap H_{\mathrm{div}}(\Omega;\mathrm{C}^n)\hookrightarrow L_2(\Omega; \mathrm{C}^n)
\end{equation*}
is compact, see \cite{P} or \cite{AxelssonMcIntosh}.

Finally, after several renumbering, we conclude that $\{(f^l,\nabla_0 g^l,kg^l)\}_{l=1}^{+\infty}$ contains a Cauchy subsequence in $\left(\mathbf{R}(D),\|\cdot\|_{L_2}\right)$.
\end{proof}

Further on, we prove similar results, but for the operator $D$ defined in Section \ref{Maxwell subsection}.

\begin{prop}\label{Max.closed range}
Let $\Omega\subset \mathrm{R}^2$ be a bounded weakly Lipschitz domain and $D$ be the operator defined in Section \ref{Maxwell subsection}. Then the range $\mathbf{R}(D)$ is a closed subspace of $L_2(\Omega; \mathrm{C}^6)$.
\end{prop}
\begin{proof}
  Let $h\in \mathbf{D}(D^*)$ and $D^*h\in\mathbf{D}(D)$. Then
  \begin{equation*}
    \|DD^*h\|^2=\left\|D\begin{bmatrix}
                          \text{div}_0h_2 \\
                          -\nabla h_1+\overline{i\omega}Jh_4\\
                          \text{div}h_4 \\
                          -\nabla_0 h_3-\overline{i\omega}Jh_2
                        \end{bmatrix}\right\|^2=\left\| \begin{matrix}
                                                          -\text{div}_0\nabla h_1+\overline{i\omega}\text{div}_0Jh_4 \\
                                                          -\nabla\text{div}_0h_2-i\omega J\nabla_0 h_3+|\omega|^2h_2 \\
                                                          -\text{div}\nabla_0 h_3-\overline{i\omega}\text{div}Jh_2 \\
                                                          -\nabla_0\text{div}h_4+i\omega J\nabla h_1+|\omega|^2h_4
                                                        \end{matrix}\right\|^2
  \end{equation*}
  \begin{equation*}
    \geq\|-\nabla\text{div}_0h_2-i\omega J\nabla_0 h_3+|\omega|^2h_2\|^2+\|-\nabla_0\text{div}h_4+i\omega J\nabla h_1+|\omega|^2h_4\|^2
  \end{equation*}
  \begin{equation*}
    =\|\nabla\text{div}_0h_2\|^2+\|\omega J\nabla_0 h_3\|^2+|\omega|^4\|h_2\|^2
  \end{equation*}
  \begin{equation*}
    +2\text{Re}(\nabla\text{div}_0h_2,i\omega J\nabla_0 h_3)+2\text{Re}(-\nabla\text{div}_0h_2,|\omega|^2h_2)+2\text{Re}(-i\omega J\nabla_0 h_3,|\omega|^2h_2)
  \end{equation*}
  \begin{equation*}
    +\|\nabla_0\text{div}h_4\|^2+\|\omega J\nabla h_1\|^2+|\omega|^4\|h_4\|^2
  \end{equation*}
  \begin{equation*}
    +2\text{Re}(-\nabla_0\text{div}h_4,i\omega J\nabla h_1)+2\text{Re}(-\nabla_0\text{div}h_4,|\omega|^2h_4)+2\text{Re}(i\omega J\nabla h_1,|\omega|^2h_4)
  \end{equation*}
\begin{equation*}
  =\|\nabla\text{div}_0h_2\|^2+|\omega|^2\|\nabla_0 h_3\|^2+|\omega|^2\|\overline{\omega}Jh_2\|^2
\end{equation*}
\begin{equation*}
  +0+2|\omega|^2\|\text{div}_0h_2\|^2+2|\omega|^2\text{Re}(-\nabla_0 h_3,-\overline{i\omega}Jh_2)
\end{equation*}
\begin{equation*}
  +\|\nabla_0\text{div}h_4\|^2+|\omega|^2\|\nabla h_1\|^2+|\omega|^2\|\overline{\omega}Jh_4\|^2
\end{equation*}
\begin{equation*}
  +0+2|\omega|^2\|\text{div}h_4\|^2+2|\omega|^2\text{Re}(\nabla h_1,-\overline{i\omega}Jh_4).
\end{equation*}
  Thus
  \begin{equation*}
    \|DD^*h\|^2\geq \|\nabla\text{div}_0h_2\|^2+\|\nabla_0\text{div}h_4\|^2+|\omega|^2\|D^*h\|^2\geq |\omega|^2\|D^*h\|^2.
  \end{equation*}
Since $\omega\neq0$, the rage $\mathbf{R}(D)$ is closed.
\end{proof}

The following proposition gives the exact expression for the range $\mathbf{R}(D)$.

\begin{prop}\label{Max H equals RD}
  Let $\Omega\subset \mathrm{R}^2$ be a bounded weakly Lipschitz domain and $D$ be the operator defined in Section \ref{Maxwell subsection}, then $\mathbf{R}(D)=\mathcal{H}$, where
  \begin{equation*}
    \mathcal{H}:=\left\{(f_{\perp},f_{\parallel},g_{\perp},g_{\parallel})\in L_2(\Omega; \mathrm{C}^6): \;
    f_{\parallel}\in H_{\mathrm{curl}}(\Omega;\mathrm{C}^2), \; g_{\parallel}\in H_{\mathrm{curl}}^0(\Omega;\mathrm{C}^2)\right.
  \end{equation*}
  \begin{equation*}
    \left.\mathrm{and} \; \mathrm{curl}f_{\parallel}-i\omega g_{\perp}=0, \; \mathrm{curl}g_{\parallel}+i\omega f_{\perp}=0\right\}.
  \end{equation*}
\end{prop}

\begin{proof}
  Assume $(f,g)\in \mathbf{R}(D)$, then there exists $(F,G)\in \mathbf{D}(D)$ such that
  \begin{equation*}
    \begin{bmatrix}
      f_{\perp} \\
      f_{\parallel} \\
      g_{\perp} \\
      g_{\parallel}
    \end{bmatrix}
    =
    \begin{bmatrix}
      \mathrm{div}_0 F_{\parallel} \\
      -\nabla F_{\perp}+i\omega JG_{\parallel} \\
      \mathrm{div} G_{\parallel} \\
      -\nabla_0 G_{\perp}-i\omega JF_{\parallel}
    \end{bmatrix}.
  \end{equation*}
  Since $F_{\parallel}\in H^0_{\mathrm{div}}(\Omega;\mathrm{C}^2)$, $G_{\parallel}\in H_{\mathrm{div}}(\Omega;\mathrm{C}^2)$, we see $f_{\perp}$, $g_{\perp}\in L_2(\Omega)$. From Remark \ref{WHdiv equals Hcurl}, we conclude $JF_{\parallel}\in H^0_{\mathrm{curl}}(\Omega;\mathrm{C}^2)$ and $JG_{\parallel}\in H_{\mathrm{curl}}(\Omega;\mathrm{C}^2)$. Therefore, since
  $F_{\perp}\in H^1(\Omega)$ and $G_{\perp}\in H^1_0(\Omega)$, we obtain $f_{\parallel}\in H_{\mathrm{curl}}(\Omega;\mathrm{C}^2)$ and $g\in H^0_{\mathrm{curl}}(\Omega;\mathrm{C}^2)$.

  Next, we compute
  \begin{equation*}
    \mathrm{curl}f_{\parallel}=-\mathrm{curl}\nabla F_{\perp}+i\omega\mathrm{curl}JG_{\parallel}=
    i\omega\mathrm{div}G_{\parallel}=i\omega g_{\perp}
  \end{equation*}
  and similarly
  \begin{equation*}
    \mathrm{curl}g_{\parallel}=-i\omega f_{\perp}.
  \end{equation*}
  From the arguments above, we can assert that $\mathbf{R}(D)\subset \mathcal{H}$.

  Conversely, assume $(f,g)\in \mathcal{H}$. Let us set
  \begin{equation*}
    F_{\parallel}=\frac{1}{i\omega}Jg_{\parallel},
    \qquad
    G_{\parallel}=\frac{1}{i\omega}Jf_{\parallel}.
  \end{equation*}
  Then, from Remark \ref{WHdiv equals Hcurl}, we obtain
  \begin{equation*}
    F_{\parallel}\in H_{\mathrm{div}}^0(\Omega;\mathrm{C}^2),
    \qquad
    G_{\parallel}\in H_{\mathrm{div}}(\Omega;\mathrm{C}^2)
  \end{equation*}
  and
  \begin{equation*}
    f_{\perp}=\mathrm{div}_0 F_{\parallel},
    \qquad
    g_{\perp}=\mathrm{div} G_{\parallel}.
  \end{equation*}

  Next, since
  \begin{equation*}
    \mathrm{curl}\left(f_{\parallel}-i\omega J G_{\parallel}\right)=\mathrm{curl}f_{\parallel}-i\omega\mathrm{div} G_{\parallel}=
    \mathrm{curl}f_{\parallel}-i\omega g_{\perp}=0
  \end{equation*}
  and $f_{\parallel}-i\omega J G_{\parallel}\in H_{\mathrm{curl}}(\Omega;\mathrm{C}^2)$, there exists a function $F_{\perp}\in H^1(\Omega)$ such that $-\nabla F_{\perp}=f_{\parallel}-i\omega J G_{\parallel}$.

  Likewise, since $\mathrm{curl}\left(g_{\parallel}+i\omega J F_{\parallel}\right)=0$ and $g_{\parallel}+i\omega J F_{\parallel}\in H_{\mathrm{curl}}^0(\Omega;\mathrm{C}^2)$, there exists a function $G_{\perp}\in H^1_0(\Omega)$ such that $-\nabla_0G_{\perp}=g_{\parallel}+i\omega J F_{\parallel}$.

  Combining all relations between $(f,g)$ and $(F,G)$, we get $(F,G)\in \mathbf{D}(D)$ and
  \begin{equation*}
    D\begin{bmatrix}
       F \\
       G
     \end{bmatrix}
     =
     \begin{bmatrix}
       f \\
       g
     \end{bmatrix}.
  \end{equation*}
  This implies that $\mathcal{H}\subset \mathbf{R}(D)$, hence that $\mathcal{H}=\mathbf{R}(D)$.
\end{proof}

There is also the following analogue of Proposition \ref{Helm.compact resolvent}.

\begin{prop}\label{Max.compact resolvent}
Let $\Omega\subset \mathrm{R}^2$ be a bounded weakly Lipschitz domain and $D$ be the operator defined in section \ref{Maxwell subsection}. Assume $\lambda\in\rho(\left.D\right|_{\mathbf{R}(D)})$, then
$$
(\lambda-\left.D\right|_{\mathbf{R}(D)})^{-1}:\mathbf{R}(D)\rightarrow \mathbf{R}(D)
$$
is a compact operator.
\end{prop}

\begin{proof}
As in Proposition \ref{Helm.compact resolvent}, we see that
\begin{equation*}
  \left.D\right|_{\mathbf{R}(D)}(\lambda-\left.D\right|_{\mathbf{R}(D)})^{-1}:\mathbf{R}(D)\rightarrow \mathbf{R}(D)
\end{equation*}
is a bounded operator. Therefore it remains to verify that the embedding
$$
\left(\mathbf{D}(D)\cap \mathbf{R}(D),\|\cdot\|_{\mathbf{R}(D)}\right)\hookrightarrow\left(\mathbf{R}(D),\|\cdot\|_{L_2}\right).
$$
is compact

  Let $\{h^l\}_{l=1}^{\infty}\subset \mathbf{D}(D)$ be a sequence such that
  $\{Dh^l\}_{l=1}^{\infty}\subset\mathbf{D}(D)\cap \mathbf{R}(D)$ and
  \begin{equation}\label{Max.compact resolvent1}
    \|Dh^l\|+\|DDh^l\|<C
  \end{equation}
  for some constant $C>0$. In particular
  \begin{equation*}
    \|\mathrm{div}_0h_2^l\|+\|-\nabla\text{div}_0 h_2^l+i\omega J(-\nabla_0 h^l_3-i\omega Jh_2^l)\|<C,
  \end{equation*}
  \begin{equation*}
    \|-\nabla_0 h^l_3-i\omega Jh_2^l\|<C.
  \end{equation*}
  Therefore
  \begin{equation}\label{Max.compact resolvent2}
    \|\nabla\mathrm{div}_0h_2^l\|+\|\mathrm{div}_0h_2^l\|<C.
  \end{equation}
  As in Proposition \ref{Helm.compact resolvent}, \eqref{Max.compact resolvent2} implies that $\{\mathrm{div}_0h^l_2\}_{l=1}^{\infty}$ contains a Cauchy subsequence in $L_2(\Omega)$. Similarly, this statement holds for $\{\mathrm{div}h^l_4\}_{l=1}^{\infty}$.

  Since $\|DDh^l\|\leq C$, we obtain
  \begin{equation*}
    \|\mathrm{div}(-\nabla_0 h^l_3-i\omega Jh_2^l)\|\leq C
  \end{equation*}
  and
  \begin{equation*}
    \|-\mathrm{curl}\nabla_0 h^l_3-i\omega\mathrm{curl}Jh_2^l)\|=\|i\omega\mathrm{curl}Jh_2^l)\|=\|i\omega\mathrm{div}h_2^l\|\leq C.
  \end{equation*}
  Therefore, $\{-\nabla_0 h^l_3-i\omega Jh_2^l\}_{l=1}^{\infty}$ is bounded in $H_{\mathrm{curl}}^{0}(\Omega;\mathrm{C}^2)\cap H_{\mathrm{div}}(\Omega;\mathrm{C}^2)$. From the compact embedding
  \begin{equation*}
    H_{\mathrm{curl}}^{0}(\Omega;\mathrm{C}^2)\cap H_{\mathrm{div}}(\Omega;\mathrm{C}^2)\hookrightarrow L_2(\Omega;\mathrm{C}^2),
  \end{equation*}
  see \cite{P} or \cite{AxelssonMcIntosh}, we conclude that $\{-\nabla_0 h^l_3-i\omega Jh_2^l\}_{l=1}^{\infty}$ contains a Cauchy subsequence in $L_2(\Omega;\mathrm{C}^2)$.

  Likewise, $\{-\nabla h^l_1+i\omega Jh_4^l\}_{l=1}^{\infty}$ is bounded in $H_{\mathrm{curl}}(\Omega;\mathrm{C}^2)\cap H_{\mathrm{div}}^{0}(\Omega;\mathrm{C}^2)$. Since $H_{\mathrm{curl}}(\Omega;\mathrm{C}^2)\cap H_{\mathrm{div}}^{0}(\Omega;\mathrm{C}^2)$ is compactly embedded to $L_2(\Omega;\mathrm{C}^2)$, see \cite{P} or \cite{AxelssonMcIntosh}, $\{-\nabla h^l_1+i\omega Jh_4^l\}_{l=1}^{\infty}$ contains a convergent subsequence in $L_2(\Omega;\mathrm{C}^2)$.

  From arguments above, we conclude that $\{Dh^l\}_{l=1}^{\infty}$ contains a Cauchy subsequence in $L_2(\Omega;\mathrm{C}^6)$.
\end{proof}

\section{Spectral projections and functional calculus for DB}\label{SP and FC} In this section we modify the functional calculus designed by McIntosh in \cite{McIntosh1986}, for the operators described below.

Let $\Omega\subset \mathrm{R}^n$ be a bounded weakly Lipschitz domain. From now on we consider a pointwise accretive multiplication operator $B\in L_{\infty}(\Omega; \mathrm{C}^M\times\mathrm{C}^M)$ on $L_2(\Omega;\mathrm{C}^M)$ and a closed range operator
$$
D:L_2(\Omega;\mathrm{C}^M)\rightarrow L_2(\Omega;\mathrm{C}^M)
$$
satisfying the following conditions
\begin{enumerate}
  \item There exist a bounded operator $D_0$ and a self-adjoint homogeneous first order differential operator $D_1$ with constant coefficients and local boundary conditions so that
  \begin{equation*}
    D=D_1+D_0.
  \end{equation*}
  \item The operator $(\lambda-\left.D\right|_{\mathbf{R}(D)})^{-1}$ is compact for some, and therefore for all $\lambda$ belonging to the resolvent set $\rho(\left.D\right|_{\mathbf{R}(D)})$.
\end{enumerate}

\begin{rem}
In both the Helmholtz and the Maxwell's cases, the operators $B$ and $D$ satisfy conditions above. Moreover, $D_0$ is a normal operator, and hence $D$ is normal as well.
\end{rem}

\subsection{Preliminary for functional calculus}\label{Preliminary for functional calculus} Here we consider basic properties of the operator $DB$ in order to construct a functional calculus in the next subsections. We begin with a well known result and give its prove for sake of completeness.

\begin{prop}\label{Gen splitting for L2}
We have topological splittings for $L_2(\Omega;\mathrm{C}^M)$
$$
L_2(\Omega;\mathrm{C}^M)=\mathbf{N}(D^*B)\oplus \mathbf{R}(D)
$$
and
$$
L_2(\Omega;\mathrm{C}^M)=\mathbf{N}(D^*)\oplus B\mathbf{R}(D).
$$
\end{prop}

\begin{proof}
Since $\mathbf{N}(B^*D^*)=\mathbf{N}(D^*)$, $\mathbf{R}(DB)=\mathbf{R}(D)$ and $B^*D^*=(DB)^*$, we obtain the following orthogonal splitting
\begin{equation*}
  L_2(\Omega;\mathrm{C}^M)=\mathbf{R}(DB)\oplus \mathbf{N}(B^*D^*)=\mathbf{R}(D)\oplus \mathbf{N}(D^*).
\end{equation*}

For nonzero $g\in \mathbf{N}(D^*)$, $(B^{-1}g,g)\neq 0$. Thus $\mathbf{R}(DB)\cap B\mathbf{N}(D^*)=\{0\}$. Since $B^*$ is an accretive operator, for $g\in \mathbf{R}(D)$ and $h\in \mathbf{N}(D^*)$, we obtain
\begin{equation}\label{max sliting eq1}
C^{-1}\|g\|^2+0\leq\text{Re}(B^{*}g,g)+\text{Re}(g,h)=\text{Re}(B^{*}g,g)+\text{Re}(B^{*}g,B^{-1}h)
\end{equation}
\begin{equation*}
=\text{Re}(B^{*}g,g+B^{-1}h)\leq C\|g\|\|g+B^{-1}h\|
\end{equation*}
for some $C>0$. Similarly
\begin{equation}\label{max sliting eq2}
C^{-1}\|B^{-1}h\|^2\leq\text{Re}(B^{*}B^{-1}h,B^{-1}h)=\text{Re}(B^{-1}h,h)
\end{equation}
\begin{equation*}
=\text{Re}(B^{-1}h+g,h)\leq C\|B^{-1}h+g\|\|h\|.
\end{equation*}
for some constant $C>0$. Therefore $B^{-1}\mathbf{N}(D^*)\oplus \mathbf{R}(D)$ is a Hilbert space. Assume $f\in(B^{-1}\mathbf{N}(D^*)\oplus \mathbf{R}(D))^{\perp}$. In particular, $f\in \mathbf{N}(D^*)$ and $f\perp \mathbf{R}(D)$. Since $B$ is an accretive operator, we see that $f=0$. Therefore
\begin{equation*}
  L_2(\Omega;\mathrm{C}^M)=B^{-1}\mathbf{N}(D^*)\oplus \mathbf{R}(D)=\mathbf{N}(D^*B)\oplus \mathbf{R}(D).
\end{equation*}

Similarly, one can prove the second splitting.
\end{proof}

\begin{prop}\label{closed and densely}
The operator
$$
\left.DB\right|_{\mathbf{R}(D)}:\mathbf{R}(D)\rightarrow \mathbf{R}(D)
$$
is a closed and densely defined operator.
\end{prop}

\begin{proof}
Note that $\mathbf{N}(D^*B)\subset\mathbf D(DB)$. Therefore, from Proposition \ref{Gen splitting for L2}, we obtain
\begin{equation}\label{splitting of domain}
  \mathbf D(DB)=[\mathbf D(DB)\cap \mathbf{R}(D)]\oplus \mathbf{N}(D^*B).
\end{equation}
Let us fix $\varepsilon>0$ and $f\in \mathbf{R}(D)$. Since $B$ is an invertible bounded operator and $\mathbf{D}(D)$ is a dense set in $L_2(\Omega;\mathrm{C}^M)$, we deduce that $\mathbf{D}(DB)=B^{-1}\mathbf{D}(D)$ is dense in $L_2(\Omega;\mathrm{C}^M)$. Therefore, from \eqref{splitting of domain}, we can find $g\in \mathbf D(DB)\cap \mathbf{R}(D)$ and $h\in \mathbf{N}(D^*B)$ such that $\|g+h-f\|\leq\varepsilon$. On the other hand, Proposition \ref{Gen splitting for L2} gives
$$
\|g+h-f\|\geq C(\|g-f\|+\|h\|).
$$
Hence $\|g-f\|\leq\frac{\varepsilon}{C}$, that is the set $\mathbf D(DB)\cap \mathbf{R}(D)$ is dense in $\mathbf{R}(D)$.

The operator
\begin{equation*}
  DB: L_2(\Omega;\mathrm{C}^M)\rightarrow L_2(\Omega;\mathrm{C}^M)
\end{equation*}
is closed and $\mathbf{R}(D)$ is closed in $L_2(\Omega;\mathrm{C}^M)$. Hence, the operator $\left.DB\right|_{\mathbf{R}(D)}$ is closed.
\end{proof}

To state the next proposition let us set
\begin{equation*}
  S_{\alpha,\tau}:=\{x+iy\in \mathrm{C}: |y|<|x|\tan \alpha+\tau\}
\end{equation*}
for $\alpha\in [0,\frac{\pi}{2})$ and $\tau\geq0$, and define the angle and constant of accretivity of $B$ to be
\begin{equation*}
  \omega:=\sup_{v\in \mathrm{C}^M}|\arg (Bv,v)|<\frac{\pi}{2}
  \quad
  \text{and}
  \quad
  \beta:=\inf_{v\in \mathrm{C}^M}\frac{\mathrm{Re}(Bv,v)}{\|v\|^2}
\end{equation*}
respectively.

\begin{prop}\label{Gen resolvent of DB}
  There exist constants $\tau,C>0$, depending only on $\|D_0B\|$, $\|B\|$ and $\beta$, such that $\sigma(DB)\subset S_{\omega,\tau}$ and
  \begin{equation}\label{Gen resolvent of DB0}
    \|(\lambda-DB)^{-1}\|\leq \frac{C}{\text{dist}(\lambda,S_{\omega,0})}
  \end{equation}
  for any $\lambda\notin S_{\omega,\tau}$.
\end{prop}

\begin{proof}
Since $D_1$ is self-adjoint, $D_1B$ is bisectorial, see [\cite{AAM},Proposition 3.3]. Therefore, for any $\lambda \notin S_{\omega,0}$ and $u\in \mathbf{D}(DB)$,
$$
\|(\lambda-DB)u\|\geq\|(\lambda-D_1B)u\|-\|D_0Bu\|\geq C\text{dist}(\lambda,S_{\omega,0})\|u\|-\|D_0B\|\|u\|.
$$
Thus, for sufficiently large $\tau>0$ and any $\lambda \notin S_{\omega,\tau}$,
\begin{equation*}
  \frac{C}{2}\text{dist}(\lambda,S_{\omega,0})\|u\|\geq \|D_0B\|\|u\|,
\end{equation*}
and therefore
\begin{equation}\label{Gen resolvent of DB1}
  \|(\lambda-DB)u\|\geq \frac{C}{2}\text{dist}(\lambda,S_{\omega,0})\|u\|.
\end{equation}
Hence $\lambda-DB$ is an injective operator with closed range. Next, let us consider the adjoint operator
\begin{equation*}
  (\lambda-DB)^*=\overline{\lambda}-B^*D^*=B^*(\overline{\lambda}-D^*B^*)B^{*-1}.
\end{equation*}
Similarly, we conclude that $\overline{\lambda}-D^*B^*$, and therefore $(\lambda-DB)^*$ are injective. Hence $\lambda-DB$ is a surjective operator. Thus, $\lambda \notin S_{\omega,\tau}$ is contained in the resolvent set and \eqref{Gen resolvent of DB1} implies \eqref{Gen resolvent of DB0}.
\end{proof}

Let $P_{\mathbf{R}(D)}$ and $P_{\mathbf{N}(D^*)}$ be orthogonal projections to $\mathbf{R}(D)$ and $\mathbf{N}(D^*)$ corresponding to the splitting
\begin{equation}\label{usual splitting}
  L_2(\Omega;\mathrm{C}^M)=\mathbf{R}(D)\oplus \mathbf{N}(D^*).
\end{equation}

\begin{lem}\label{invertibility of projection}
The operator
$$
\left.P_{\mathbf{R}(D)}\right|_{B\mathbf{R}(D)}:B\mathbf{R}(D)\rightarrow \mathbf{R}(D)
$$
is bounded and invertible.
\end{lem}

\begin{proof}
If $P_{\mathbf{R}(D)}BDf=0$, then $(BDf,Df)=0$. This implies that $Df=0$, and hence that $\left.P_{\mathbf{R}(D)}\right|_{B\mathbf{R}(D)}$ is an injective operator. The second splitting in Proposition \ref{Gen splitting for L2}
implies that operator $\left.P_{\mathbf{R}(D)}\right|_{B\mathbf{R}(D)}$ is surjective. Thus, by the bounded inverse theorem, we get the statement of the lemma.
\end{proof}

\begin{prop}\label{Abs.compact resolvent}
Let $\lambda\in\rho(\left.DB\right|_{\mathbf{R}(D)})$, then
$$
(\lambda-\left.DB\right|_{\mathbf{R}(D)})^{-1}:\mathbf{R}(D)\rightarrow \mathbf{R}(D)
$$
is a compact operator.
\end{prop}

\begin{proof}
As in Propositions \ref{Helm.compact resolvent} and \ref{Max.compact resolvent}, it suffices to prove that the embedding
$$
\left(\mathbf{D}(DB)\cap \mathbf{R}(D), \|\cdot\|_{\mathbf{D}(DB)\cap \mathbf{R}(DB)}\right)
\hookrightarrow\left(\mathbf{R}(D), \|\cdot\|_{L_2}\right)
$$
is compact.

Let $\{f^l\}_{l=1}^{\infty}\subset\left(\mathbf{D}(DB)\cap \mathbf{R}(D), \|\cdot\|_{\mathbf{D}(DB)\cap \mathbf{R}(DB)}\right)$ be a sequence such that
$$
\|f^l\|+\|DBf^l\|\leq C
$$
for some $C>0$. Since $\left.D\right|_{\mathbf{N}(D^{*})}$ is bounded, splitting \eqref{usual splitting} implies
$$
\|f^l\|+\|DP_{\mathbf{R}(D)}Bf^l\|\leq C,
$$
and therefore
$$
\|P_{\mathbf{R}(D)}Bf^l\|+\|DP_{\mathbf{R}(D)}Bf^l\|\leq C
$$
for some $C>0$. Since $(\lambda-\left.D\right|_{\mathbf{R}(D)})^{-1}$ is a compact operator, we see that
$$
\left(\mathbf{D}(D)\cap \mathbf{R}(D), \|\cdot\|_{\mathbf{D}(D)\cap \mathbf{R}(D)}\right)\hookrightarrow\left(\mathbf{R}(D), \|\cdot\|_{L_2}\right)
$$
is a compact embedding. Hence, the sequence $\{P_{\mathbf{R}(D)}Bf^l\}_{l=1}^{\infty}$ contains a Cauchy subsequence, and therefore Lemma \ref{invertibility of projection} implies that the sequence $\{f^l\}_{l=1}^{\infty}$ contains a Cauchy subsequence in $L_2(\Omega; \mathrm{C}^M)$ as well.
\end{proof}

We conclude this preliminary subsection by introducing the following setup. We fix constant $\tau>0$ from Proposition \ref{Gen resolvent of DB} and define
\begin{equation*}
  \mathcal{H}:=\mathbf{R}(D),
  \qquad
  T:=\left.DB\right|_{\mathcal{H}},
  \qquad
  T_1:=\left.D_1B\right|_{\mathcal{H}},
  \qquad
  T_0:=\left.D_0B\right|_{\mathcal{H}}.
\end{equation*}
By summarizing Propositions \ref{closed and densely}, \ref{Gen resolvent of DB} and \ref{Abs.compact resolvent}, we conclude that $T$ is a closed densely defined operator with $\sigma(T)\subset S_{\omega,\tau}$. Moreover, for each $\lambda\notin S_{\omega,\tau}$, the operator $(\lambda-T)^{-1}$ is compact,
and hence there may be only a finite number of eigenvalues of $T$ on the imaginary axis. We denote them by $\{\lambda_{i}^0\}_{i=1}^{N}$. We fix constants $a, R>0$, such that $R<a$ and
\begin{equation}\label{constant a}
  \sigma(T)\cap\{\zeta\in \mathrm{C}: \; |\text{Re}\zeta|\leq a\}=\{\lambda_{i}^0\}_{i=1}^{N},
\end{equation}
\begin{equation}\label{constant R}
  \{\zeta\in \mathrm{C}: |\zeta-\lambda_{i}^{0}|\leq R\}\cap \{\zeta\in \mathrm{C}: |\zeta-\lambda_{j}^{0}|\leq R\}=\emptyset
\end{equation}
and
\begin{equation*}
  \{\zeta\in \mathrm{C}: |\zeta-\lambda_{i}^{0}|\leq R\}\subset S_{\omega,\tau}
\end{equation*}
for $1\leq i<j\leq N$.

For $\mu\in(\omega,\frac{\pi}{2})$, we fix the open set
\begin{equation*}
  \Sigma:=\Sigma^-\cup\Sigma^+\cup\Sigma^0,
\end{equation*}
where
\begin{equation*}
  \Sigma^{\pm}:=\{\zeta\in \mathrm{C}: \; \pm\text{Re}\zeta>a, \; |\text{Im}\zeta|< \tau+|\text{Re}\zeta|\tan\mu\}
\end{equation*}
and
\begin{equation*}
  \Sigma^0:=\cup_{i=1}^{N}\{\zeta\in \mathrm{C}: |\zeta-\lambda_{i}^{0}|< R\}.
\end{equation*}
Due to \eqref{constant a} and \eqref{constant R}, $\Sigma$ is a disjoint union of $N+2$ open, connected sets and $\sigma(T)\subset\Sigma$.

Next we define
\begin{equation*}
  H^{\infty}(\Sigma):=\{h:\Sigma\rightarrow \mathrm{C} \; \text{holomorphic}, \; \sup_{z\in \Sigma}|h(z)|<\infty\},
\end{equation*}
\begin{equation*}
  \Theta(\Sigma):=\{\psi\in H^{\infty}(\Sigma): |\psi(z)|\leq \frac{C}{|z|^{\alpha}}, \; \text{for some} \; \alpha,C>0 \; \text{and all} \; z\in \Sigma\}.
\end{equation*}
For $b>a$ such that
\begin{equation}\label{FC constant b}
  \sigma(T)\cap\{\zeta\in \mathrm{C}: \; a\leq|\text{Re}\zeta|\leq b\}=\emptyset
\end{equation}
and $\nu\in(\omega,\mu)$, $r<R$, we define anti-clockwise oriented curves
\begin{equation*}
  \gamma^{\pm}:=\{\zeta\in \mathrm{C}: \; \pm\text{Re}\zeta= b, \; |\text{Im}\zeta|\leq \tau+|\text{Re}\zeta|\tan\nu\}
\end{equation*}
\begin{flushright}
\begin{equation*}
  \bigcup \{\zeta\in \mathrm{C}: \; \pm\text{Re}\zeta>b, \; \text{Im}\zeta= \tau+|\text{Re}\zeta|\tan\nu\}
\end{equation*}
\end{flushright}
\begin{flushright}
\begin{equation*}
  \bigcup \{\zeta\in \mathrm{C}: \; \pm\text{Re}\zeta>b, \; \text{Im}\zeta=-\left(\tau+|\text{Re}\zeta|\tan\nu\right)\},
\end{equation*}
\end{flushright}
\begin{equation}\label{FC definition of gamma0}
  \gamma^0:=\bigcup_{i=1}^N\{\zeta\in \mathrm{C}: \; |\zeta-\lambda_{i}^{0}|=r\}
\end{equation}
and
\begin{equation}\label{definition of gamma}
  \gamma:=\gamma^-\cup\gamma^+\cup\gamma^0.
\end{equation}
See figure 1.

\begin{figure}[h!]
\centering
\includegraphics[width=100mm]{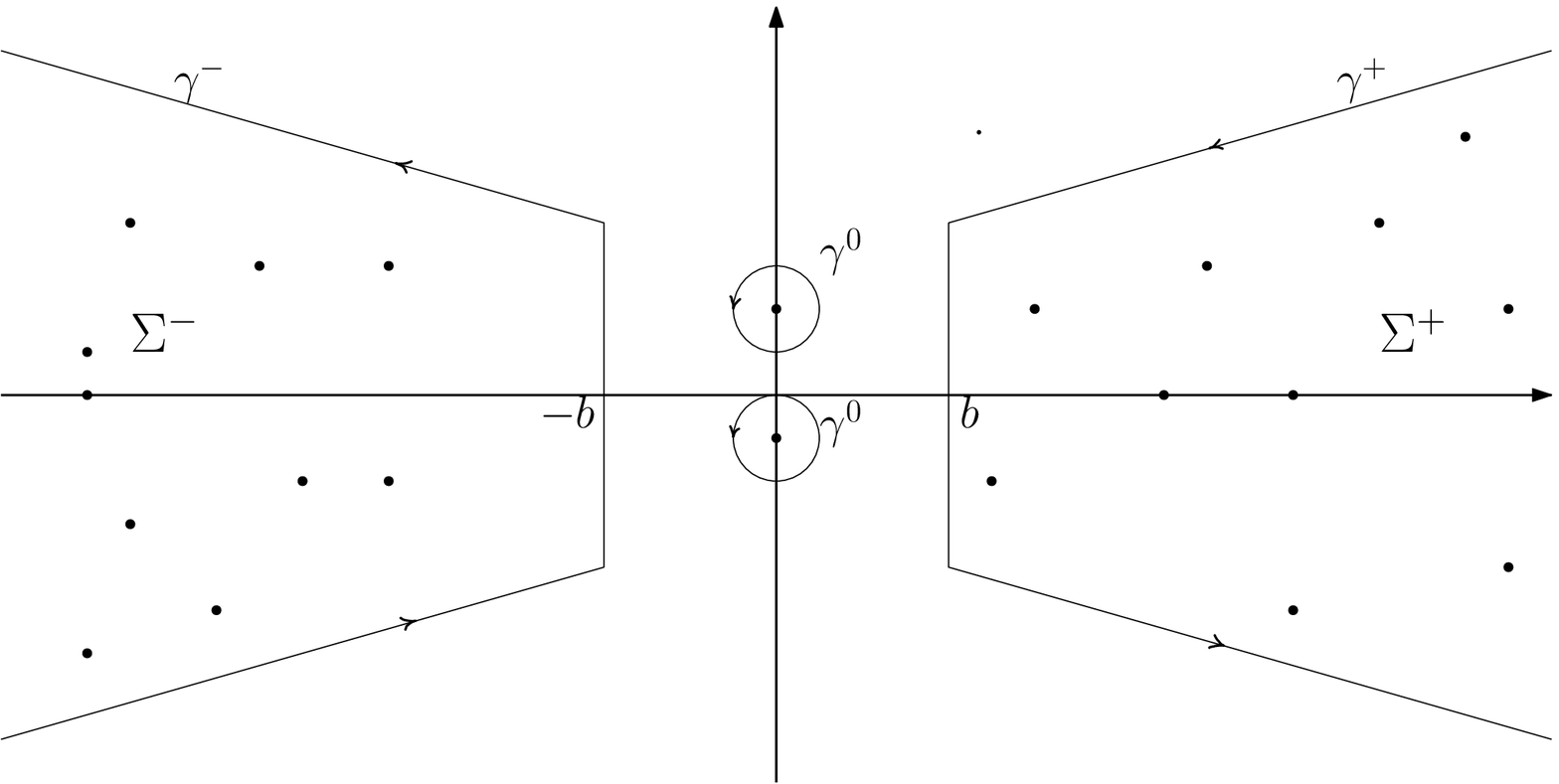}
\caption{N=2}
\end{figure}

\subsection{The $\mathbf{\Theta(\Sigma)}$ functional calculus} In this subsection we introduce the following preliminary functional calculus.

\begin{defn}\label{definition of psi(T)}
Let $r<R$, $0<\nu<\mu$ and $b>a$ such that \ref{FC constant b} holds. For $\psi\in \Theta(\Sigma)$, we define $\psi(T)$ by
\begin{equation}\label{definition of psi(T)1}
  \psi(T)=\frac{1}{2\pi i}\int_{\gamma}\frac{\psi(\zeta)}{\zeta-T}d\zeta,
\end{equation}
where $\gamma$ is the curve defined in \eqref{definition of gamma}.
\end{defn}

A justification of this definition follows from the next proposition.

\begin{prop}\label{projection for bisector}
For $\psi\in \Theta(\Sigma)$, the integral
\begin{equation*}
  \frac{1}{2\pi i}\int_{\gamma}\frac{\psi(\zeta)}{\zeta-T}d\zeta
\end{equation*}
converges absolutely. Moreover, the integral is independent of the choice of\\ $\gamma=\gamma^{\pm}(r,\nu,b)$, where $\nu\in (\omega,\mu)$ and $0<r<R$, $b>a$ such that \eqref{FC constant b} holds.
\end{prop}

\begin{proof}
  We give only the main ideas of the proof. For $\psi\in \Theta(\Sigma)$, Proposition \ref{Gen resolvent of DB} implies
  \begin{equation*}
    |\psi(\zeta)|\|(\zeta-T)^{-1}\|\leq C\frac{1}{|\zeta|}\frac{1}{|\zeta|^{\alpha}}.
  \end{equation*}
  Therefore, the first statement follows from the convergence
  \begin{equation*}
    \int_{\varepsilon}^{+\infty}\frac{1}{x^{\alpha+1}}dx<\infty
  \end{equation*}
  for $\varepsilon>0$.

  Next, let us prove that the integral is independent of the choice of $\nu$. Assume $\omega<\nu_1<\nu_2<\mu$. For $P>0$, we set
  \begin{equation*}
    \delta_P^+(t):= b\pm i(b\tan \nu +\tau)+ Pe^{\pm i(t\nu_2+(1-t)\nu_1)}.
  \end{equation*}
  Then
  \begin{equation*}
    \left\|\int_{\delta_P^{\pm}}\frac{\psi(\zeta)}{\zeta-T}d\zeta\right\|\leq C l(\delta_P^{\pm})\frac{1}{P^{\alpha}}\frac{1}{P}\leq C\frac{1}{P^{\alpha}},
  \end{equation*}
  where $l(\delta_P^{\pm})$ is the length of $\delta_P^{\pm}$. Letting $P\rightarrow \infty$, we obtain independence on the choice of $\nu$.

  Finally, suppose $b_1$ and $b_2$ satisfy appropriate assumptions of the proposition and $b_1<b_2$. Then, there is no spectrum point inside the region $b_1\leq \mathrm{Re}\lambda\leq b_2$. This shows that integral is independent of choice of $b$.
\end{proof}

The proofs of the next three propositions are standard and based on proofs for bisectorial operators, see for instance \cite{AlbrechtDuongMcIntosh}, \cite{AuscherAxelsson2011}. First we prove that the map given by \eqref{definition of psi(T)1} is an algebra homomorphism.

\begin{prop}\label{fg equal fg}
If $\psi_1,\psi_2\in \Theta(\Sigma)$, then
\begin{equation*}
  \psi_1(T)+\psi_2(T)=(\psi_1+\psi_2)(T)
\end{equation*}
and
\begin{equation*}
  \psi_1(T)\psi_2(T)=(\psi_1\psi_2)(T).
\end{equation*}
\end{prop}
\begin{proof}
  For $0<r_1<r_2<R$, $0<\nu_1<\nu_2<\mu$ and $b_1>b_2>a$ such that
  $$
  \sigma(T)\cup \{\zeta\in\mathbb{C}: \; a<|\text{Re}\zeta|<b_1\}=\emptyset,
  $$
  we define two curves $\gamma_1$ and $\gamma_2$ as in \eqref{definition of gamma}. Note that $\gamma_1$ belongs to the interior of $\gamma_2$. Then
  \begin{equation*}
    (2\pi i)^2\psi_1(T)\psi_2(T)=\left(\int_{\gamma_1}\frac{\psi_1(\lambda)}{\lambda-T}d\lambda \right)\left(\int_{\gamma_2}\frac{\psi_2(\zeta)}{\zeta-T}d\zeta \right)
  \end{equation*}
  \begin{equation*}
    =\int_{\gamma_1}\int_{\gamma_2}\psi_1(\lambda)\psi_2(\zeta)\frac{1}{\zeta-\lambda}\left(\frac{1}{\lambda-T}-\frac{1}{\zeta-T}\right)d\zeta d\lambda
  \end{equation*}
  \begin{equation*}
     =\int_{\gamma_1}\frac{\psi_1(\lambda)}{\lambda-T}\left(\int_{\gamma_2}\frac{\psi_2(\zeta)}{\zeta-\lambda}d\zeta\right)d\lambda-
     \int_{\gamma_2}\left(\int_{\gamma_1}\frac{\psi_1(\lambda)}{\zeta-\lambda}d\lambda\right)\frac{\psi_2(\zeta)}{\zeta-T}d\zeta.
  \end{equation*}
  Using the Cauchy formula, we see that the second term vanishes and
  \begin{equation*}
    (2\pi i)^2\psi_1(T)\psi_2(T)=2\pi i\int_{\gamma_1}\frac{\psi_1(\lambda)}{\lambda-T}\psi_2(\lambda)d\lambda=(2\pi i)^2(\psi_1\psi_2)(T).
  \end{equation*}
\end{proof}

Next we prove the convergence lemma for the $\Theta(\Sigma)$ functional calculus.

\begin{prop}\label{convergent lemma1}
Let  $\psi_n,\psi\in \Theta(\Sigma)$ for $n\in\mathrm{N}$. Assume that $\psi_n\rightarrow \psi$ uniformly on compact subsets of $\Sigma$ and there exist $n$-independent constants $\alpha>0$, $C>0$ such that
\begin{equation*}
  |\psi_n(\zeta)|<\frac{C}{|\zeta|^{\alpha}}
\end{equation*}
for $\zeta\in \Sigma$. Then $\psi_n(T)\rightarrow \psi(T)$ in the operator norm.
\end{prop}

\begin{proof}
  Let us fix $\varepsilon>0$. One can find an integer $m_1\in \mathrm{N}$ such that for any $n>m_1$,
  \begin{equation*}
    \left\|\int_{\gamma^0}\frac{\psi_n(\zeta)-\psi(\zeta)}{\zeta-T}d\zeta\right\|\leq C\|\psi_n-\psi\|_{L^{\infty}(\gamma^0)}\left\|\int_{\gamma^0}\frac{1}{\zeta-T}d\zeta\right\|\leq \frac{2\pi\varepsilon}{3}.
  \end{equation*}
  Let $\gamma_{p,q}:=\{\zeta\in \gamma: \; p\leq|\zeta|<q\}$, then we can fix $M>0$ such that
  \begin{equation*}
    \left\|\int_{\gamma_{M, \infty}}\frac{\psi_n(\zeta)-\psi(\zeta)}{\zeta-T}d\zeta\right\|\leq C\int_{M}^{+\infty}\frac{1}{r^{\alpha+1}}dr<\frac{2\pi\varepsilon}{3}.
  \end{equation*}
  Moreover, since $a>0$, there exists $m_2\in \mathrm{N}$ such that for any $n>m_2$,
  \begin{equation*}
    \left\|\int_{\gamma_{b, M}}\frac{\psi_n(\zeta)-\psi(\zeta)}{\zeta-T}d\zeta\right\|\leq C\|\psi_n-\psi\|_{L^{\infty}(\gamma_{b, M})}\left\|\int_{\gamma_{a, M}}\frac{1}{\zeta-T}d\zeta\right\|<\frac{2\pi\varepsilon}{3}.
  \end{equation*}
  By choosing $n>\max(m_1,m_2)$, we obtain $\|\psi_n(T)-\psi(T)\|<\varepsilon$.
\end{proof}

 The following proposition together with Proposition \ref{quadratic estimate} allow us to derive an $H^{\infty}(\Sigma)$ functional calculus from the $\Theta(\Sigma)$ functional calculus.

\begin{prop}\label{convergent lemma2}
  Let $\{f_j\}_{j=1}^{\infty}\subset\Theta(\Sigma)$ be a sequence such that $\|f_j\|_{L^{\infty}(\Sigma)}<C$ and $\|f_j(T)\|<C$ for all $j\in \mathrm{N}$ and some $C>0$. Assume $f\in H^{\infty}(\Sigma)$ and $f_j\rightarrow f$ uniformly on compact subsets of $\Sigma$. Then, for any $u\in \mathcal{H}$, the sequence $\{f_j(T)u\}_{j=1}^{\infty}$ is convergent in $\mathcal{H}$. Moreover, if $f(z)=1$ on $\Sigma$, then $f_j(T)u\rightarrow u$ in $\mathcal{H}$.
\end{prop}

\begin{proof}
Let $\tau_1>\tau$ and $u\in \mathbf{D}(T)$. Since $i\tau_1\notin S_{\omega,\tau}$, there exists $v\in \mathcal{H}$ such that
$$
u=(i\tau_1-T)^{-1}v.
$$
Let $\psi(z)=f(z)\frac{1}{i\tau_1-z}$ and $\psi_j(z)=f_j(z)\frac{1}{i\tau_1-z}$ on $\Sigma$. By Proposition \ref{fg equal fg}, we see $f_j(T)u=\psi_j(T)v$, and therefore Proposition \ref{convergent lemma1} implies that $\{f_j(T)u\}_{j=1}^{\infty}$ converges to $\psi(T)v$ in $\mathcal{H}$.

Next, let $u\in \mathcal{H}$. Since $\mathbf{D}(T)$ is a dense set in $\mathcal{H}$, there exists a sequence $\{u_k\}_{k=1}^{\infty}\subset \mathbf{D}(T)$ converging to $u$ in $\mathcal{H}$. Thus
\begin{equation*}
  \|f_m(T)u-f_n(T)u\|\leq\|(f_m(T)-f_n(T))(u-u_k)\|+\|f_m(T)u_k-f_n(T)u_k\|
\end{equation*}

\begin{equation*}
  \leq2C\|u-u_k\|+\|(f_m(T)-f_n(T))u_k\|.
\end{equation*}
By choosing $k$ large enough and then letting $m,n\rightarrow\infty$, we conclude that $\{f_j(T)u\}_{j=1}^{\infty}$ is a Cauchy sequence.

Finally, if $f(z)=1$ on $\Sigma$ and $u\in \mathbf{D}(T)$, then arguments above imply that $f_j(T)u\rightarrow u$ in $\mathcal{H}$. For $u\in \mathcal{H}$, there exists a sequence $\{u_k\}_{k=1}^{\infty}\subset \mathbf{D}(T)$ converging to $u$ in $\mathcal{H}$. Thus
\begin{equation*}
  \|f_j(T)u-u\|=\|f_j(T)u-f_j(T)u_k\|+\|u_k-u\|+\|f_j(T)u_k-u_k\|.
\end{equation*}
By choosing $k$ large enough and then letting $j\rightarrow\infty$, we get $f_j(T)u\rightarrow u$ in $\mathcal{H}$.
\end{proof}

\begin{rem}\label{remark convergent lemma2}
  Note that we do not use the uniform boundedness of the sequence  $\left\{f_k(T)\right\}_{k=1}^{\infty}$ to prove the second part of Proposition \ref{convergent lemma2}.
\end{rem}

\begin{defn}
For an eigenvalue $\lambda\in \sigma(T)$, define the index of $\lambda$, as the smallest nonnegative integer $m$ such that
\begin{equation*}
  \mathbf{N}((\lambda-T)^{m})=\mathbf{N}((\lambda-T)^{m+1}).
\end{equation*}
\end{defn}

Next we prove that all purely imaginary eigenvalues of $T$ have finite index.

\begin{prop}
  The index $m_i$ of $\lambda_i^0$ is a finite number for $i=1,...,N$.
\end{prop}

\begin{proof}
  Let us set
  \begin{equation*}
    p_i(z)=\begin{cases}
             1, & \mbox{if } |z-\lambda_i^0|\leq R, \\
             0, & \mbox{otherwise}.
           \end{cases}
  \end{equation*}
  Since $p_i\in \Theta(\Sigma)$, we can define $\Pi_i:=p_i(T)$ for $i=1,...N$. By Proposition \ref{Abs.compact resolvent}, $(\lambda-T)^{-1}$ is a compact operator for all $\lambda\in \rho(T)$. Hence $\Pi_i$ is a compact operator as the Riemann sum of compact operators. Moreover, Proposition \ref{fg equal fg} implies that $\Pi_i$ is a projection. Therefore $\Pi_i$ is a finite rank operator and
  \begin{equation}\label{H splitting Pi}
    \mathcal{H}=\mathbf{N}(\Pi_i)\oplus\mathbf{R}(\Pi_i).
  \end{equation}

  Finally, for any integer $m>0$, we obtain $\mathbf{N}((\lambda_i^0-T)^m)\subset \mathbf{R}(\Pi_i)$. Therefore, the index of $\lambda_i^0$ is a finite number.
\end{proof}

We conclude this subsection with the following inequality, which will be used in Section \ref{Applications}.

\begin{prop}\label{algebraic multiplicity}
For fixed $i=1,...,N$, there exists a constant $C>0$ such that for all $h\in H^{\infty}(\Sigma)$ satisfying $h(z)=0$ for $z\notin \{\zeta\in\mathrm{C}: \; |\lambda_i^0-\zeta|<R\}$, the following estimate holds
\begin{equation*}
  \|h(T)\|\leq C\max_{0\leq j\leq m_i-1}|h^{(j)}(\lambda^{0}_{i})|.
\end{equation*}
\end{prop}

\begin{proof}
  From the assumption, $h(T)u=h(T)\Pi_0u=0$ for $u\in \mathbf{N}(\Pi_{i})$. Therefore, due to \eqref{H splitting Pi}, it suffices to prove
  \begin{equation}\label{local functional boundness}
  \|h(T)v\|\leq C\max_{0\leq j< m_i}|h^{(j)}(\lambda^{0}_{i})|\|v\|
  \end{equation}
  for all $v\in \mathbf{R}(\Pi_{i})$ and some $C>0$.

  Since $\left.T\right|_{\mathbf{R}(\Pi_{i})}$ is bounded and $\mathbf{R}(\Pi_{i})=\mathbf{N}((\lambda-T)^{m_i})$, we obtain
  \begin{equation*}
    h(T)v=\sum_{k=0}^{m_i-1}\frac{h^{(k)}(\lambda_i^0)}{k!}(\lambda_i^0-T)^{k}v
  \end{equation*}
  for any $v\in \mathbf{R}(\Pi_{i})$. This implies \eqref{local functional boundness}.
\end{proof}

\subsection{The $\mathbf{H^{\infty}(\Sigma)}$ functional calculus} Here we prove that $T$ has a bounded $H^{\infty}(\Sigma)$ functional calculus. In order to do this, analogously to functional calculus for bisectorial operators, we need the following quadratic estimate.

\begin{prop}\label{quadratic estimate}
There exists a constant $C>0$ such that
\begin{equation}\label{quadratic estimate1}
\int_{0}^{\frac{1}{\tau}}\left\|\frac{tT}{1+t^2T^2}u\right\|^2\frac{dt}{t}\leq C\|u\|^2
\end{equation}
for all $u\in \mathcal{H}$.
\end{prop}

\begin{proof}
Note that $\pm\frac{i}{t}\notin S_{\omega,\tau}$ for $t\in(0,\tau)$. Hence, by Proposition \ref{Gen resolvent of DB}, we obtain
$$
\|(1+itT)^{-1}-(1+itT_1)^{-1}\|=\|(1+itT)^{-1}(tT_0)(1+itT_1)^{-1}\|\leq C|t|.
$$
Thus
\begin{equation}\label{quadratic estimate2}
  \|tT(1+t^2T^2)^{-1}-tT_1(1+t^2T_1^2)^{-1}\|
\end{equation}

$$
=\frac{1}{2i}\|(1+itT)^{-1}-(1-itT)^{-1}+(1+itT_1)^{-1}-(1-itT_1)^{-1}\|\leq C|t|.
$$
The quadratic estimate \eqref{quadratic estimate1} for $T_1$ was proved in [\cite{AKM},Theorem 3.1]. Therefore \eqref{quadratic estimate2} implies \eqref{quadratic estimate1}.
\end{proof}

Next we prove the following auxiliary lemma.

\begin{lem}\label{P plus Q equal 1}
Let $P$, $Q$ be the operators defined by
$$
Pu=\frac{\tau^2}{\tau^2+T^2}u
\quad
\text{and}
\quad
Qu=2\int_{0}^{\frac{1}{\tau}}\left(sT\frac{1}{1+s^2T^2}\right)^2u\frac{ds}{s}
$$
for $u\in \mathcal{H}$. Then the following identity
$$
(P+Q)u=u
$$
holds for $u\in \mathcal{H}$.
\end{lem}

\begin{proof}
  Let us consider the functions
  \begin{equation*}
    f_m(z)=\frac{\tau^2}{\tau^2+z^2}+2\sum_{j=1}^{m}\frac{1}{j}\frac{\left(\frac{j}{\tau m}z\right)^2}{\left(1+\left(\frac{j}{\tau m}z\right)^2\right)^2}.
  \end{equation*}
  Observe that $f_m\rightarrow 1$ pointwise on $\Sigma$. Actually, $\{f\}_{m=1}^{\infty}$ converges uniformly on compact subsets of $\Sigma$. Indeed, assume there exist a compact subset $K\subset \Sigma$ and $\{x_k\}_{k=1}^{\infty}\subset K$ such that
  \begin{equation*}
    |f_m(x_m)-1|>c
  \end{equation*}
  for some $c>0$. Since $K$ is compact, without lost of generality we assume that $x_m\rightarrow x$ for some $x\in K$. Then
  \begin{equation*}
    c<|f_m(x_m)-1|<|f_m(x)-1|+|f_m(x_m)-f_m(x)|.
  \end{equation*}
  The first term tends to $0$, because of pointwise convergence. To estimate the second term, let us note that $\mathrm{dist}(i\tau, \Sigma)>0$, and hance there exists $C>0$ such that
  \begin{equation*}
    \left|\frac{1}{1+(\alpha z)^2}\right|<C
  \end{equation*}
  for any $\alpha\in [0,\frac{1}{\tau}]$, $z\in \Sigma$. Therefore, straightforward calculations give
  \begin{equation*}
    |f_m(x_m)-f_m(x)|\leq \sum_{j=1}^{m}\frac{1}{j}\left(\frac{j}{\tau m}\right)^2C|x-x_m|\leq C|x-x_m|.
  \end{equation*}
  This contradicts with our assumption, that is $f_m\rightarrow 1$ uniformly on compact subsets of $\Sigma$.
  Therefore Proposition \ref{convergent lemma2} and Remark \ref{remark convergent lemma2} imply that
  \begin{equation}\label{P plus Q equal 11}
    f_m(T)u\rightarrow u
  \end{equation}
  for all $u\in \mathcal{H}$.

  On the other hand, Proposition \ref{fg equal fg} yields
  \begin{equation*}
    f_m(T)u=\frac{\tau^2}{\tau^2+T^2}+2\sum_{j=1}^{m}\frac{1}{j}\left(\frac{j}{\tau m}T\left(1+\left(\frac{j}{\tau m}T\right)^2\right)^{-1}\right)^2u
  \end{equation*}
  for each $u\in \mathcal{H}$, and therefore
  \begin{equation*}
    f_m(T)u\rightarrow Pu+Qu.
  \end{equation*}
  Hence, due to \eqref{P plus Q equal 11}, we derive $Pu+Qu=u$.
\end{proof}

Now we prove that $T$ has a bounded $H^{\infty}(\Sigma)$ functional calculus . The main idea is contained in \cite{BandaraRosen}, \cite{AuscherAxelsson2011}.

\begin{thm}\label{functional calculus}
  There exists a constant $C>0$ such that the following estimate
  $$
  \|f(T)\|\leq C\|f\|_{\infty}
  $$
  holds for all $f\in \Theta(\Sigma)$.
\end{thm}

\begin{proof}
  Let $\psi_t(z)=\frac{tz}{1+t^2z^2}\in \Theta(\Sigma)$ and $P$, $Q$ be the operators as in Lemma \ref{P plus Q equal 1}. Then, for $v,u\in \mathcal{H}$,
  $$
  |(v,f(T)u)|=|(v,(P+Q)f(T)(P+Q)u)|
  $$
  $$
  \leq|(v,Pf(T)Pu)|+|(v,(I-P)f(T)Pu)|
$$
$$
  +|(v,Pf(T)(I-P)u)|+|(v,Qf(T)Qu)|
  $$
  $$
  \leq 3|(v,Pf(T)Pu)|+2|(v,Pf(T)u)|+|(v,Qf(T)Qu)|.
  $$
  We estimate each summand separately. For the first two terms, using Proposition \ref{fg equal fg}, we obtain
  $$
  |(v,Pf(T)Pu)|\leq \|v\|\|u\|\left\|\int_{\gamma}\frac{\tau^4f(z)}{(\tau^2+z^2)^2}(z-T)^{-1}dz\right\|
  $$
  $$
  \leq C\|v\|\|u\|\|f\|_{\infty}
  $$
  and
  $$
  |(v,Pf(T)u)|\leq \|v\|\|u\|\left\|\int_{\gamma}\frac{\tau^2f(z)}{\tau^2+z^2}(z-T)^{-1}dz\right\|
  $$
  $$
  \leq C\|v\|\|u\|\|f\|_{\infty}.
  $$

  To estimate the last term, we note
  $$
  \|\psi_{s}(T)f(T)\psi_{t}(T)\|\leq \|f\|_{\infty}\int_{0}^{+\infty}\frac{stx^2}{(1+s^2x^2)(1+t^2x^2)}\frac{dx}{x}
  $$
  $$
  \leq \|f\|_{\infty}\min\left(\left(\frac{t}{s}\right)^{\alpha},\; \left(\frac{s}{t}\right)^{\alpha}\right)\left(1+\left|\log\left(\frac{t}{s}\right)\right|\right)
  $$
  for $t,s\in (0,\frac{1}{\tau})$ and some $\alpha>0$. Denote $\eta(x)=\min \left(x^{\alpha},\;x^{-\alpha}\right)\left(1+\left|\log x\right|\right)$. Then
  $$
  |(v,Qf(T)Qu)|\leq C\int_{0}^{\frac{1}{\tau}}\int_{0}^{\frac{1}{\tau}}
  \|\psi_{s}^*(T)v\|\|\psi_{s}(T)f(T)\psi_{t}(T)\|\|\psi_{t}(T)u\|\frac{dt}{t}
  \frac{ds}{s}
  $$
  $$
  \leq C\|f\|_{\infty}\int_{0}^{\frac{1}{\tau}}\int_{0}^{\frac{1}{\tau}}
  \|\psi_{s}^*(T)v\|\|\psi_{t}(T)u\|\eta\left(\frac{t}{s}\right)\frac{dt}{t}
  \frac{ds}{s}.
  $$
  The Cauchy-Schwartz inequality yields
  $$
  |(v,Qf(T)Qu)|^2\leq
  C\|f\|_{\infty}^2\left(\int_{0}^{\frac{1}{\tau}}\|\psi_{s}^*(T)v\|^2\left(\int_{0}^{\frac{1}{\tau}}\eta\left(\frac{t}{s}\right)\frac{dt}{t}\right)\frac{ds}{s}\right)\times
  $$
  \begin{flushright}
    $$
    \times\left(\int_{0}^{\frac{1}{\tau}}\|\psi_{t}(T)u\|^2\left(\int_{0}^{\frac{1}{\tau}}\eta\left(\frac{t}{s}\right)\frac{ds}{s}\right)\frac{dt}{t}\right).
    $$
  \end{flushright}
  Finally, using the quadratic estimate from Proposition \ref{quadratic estimate}, we get
  $$
  |(v,Qf(T)Qu)|\leq C\|f\|_{\infty}\|u\|\|v\|.
  $$
  \end{proof}

Now we are on a position to introduce the following $H^{\infty}(\Sigma)$ functional calculus for the operator $T$.

\begin{defn}\label{definition of f(T)}
Let $f\in H^{\infty}(\Sigma)$ and $\{\psi_i\}_{i=1}^{\infty}\subset \Theta(\Sigma)$ be an uniformly bounded sequence such that $\psi_i \rightarrow f$ uniformly on compact subsets of $\Sigma$. We define
\begin{equation*}
  f(T)u=\lim_{i\rightarrow \infty}\psi_i(T)u
\end{equation*}
for $u\in \mathcal{H}$.
\end{defn}

By Proposition \ref{convergent lemma2}, the definition of $f(T)$ is independent of the choice of sequence $\{\psi_i\}_{i=1}^{\infty}$. Also observe that the sequence $\{\frac{im}{im+z}f(z)\}_{m=1}^{\infty}\subset \Theta(\Sigma)$ converges to $f$ uniformly on compact subsets of $\Sigma$ for $f\in H^{\infty}(\Sigma)$. Therefore Proposition \ref{functional calculus} implies that we have a well defined bounded operator $f(T)$ on $\mathcal{H}$ for any $f\in H^{\infty}(\Sigma)$.

Proposition \ref{convergent lemma2} also shows that Definition \ref{definition of f(T)} agrees with Definition \ref{definition of psi(T)} for functions in $\Theta(\Sigma)$.

Let us consider the basic properties of the $H^{\infty}(\Sigma)$ functional calculus. First we prove that the map given by Definition \ref{definition of f(T)} is an algebra homomorphism.

\begin{prop}\label{fg equal fg H class}
  Let $f$, $g\in H^{\infty}(\Sigma)$. Then
  \begin{equation*}
    f(T)+g(T)=(f+g)(T)
  \end{equation*}
  and
  \begin{equation*}
    f(T)g(T)=(fg)(T).
  \end{equation*}
\end{prop}

\begin{proof}
  Let $f$, $g\in H^{\infty}(\Sigma)$ and $\{f_j\}_{j=1}^{\infty}$, $\{g_j\}_{j=1}^{\infty}\subset\Theta(\Sigma)$ be the corresponding sequences, see Definition \ref{definition of f(T)}. Then $\{fg_j\}_{j=1}^{\infty}\subset\Theta(\Sigma)$ is uniformly bounded and $fg_j\rightarrow fg$ on compact subsets of $\Sigma$. Therefore
  \begin{equation}\label{fg equal fg H class1}
    (fg)(T)u=\lim_{j\rightarrow \infty}(fg_j)(T)u
  \end{equation}
  for each $u\in \mathcal{H}$. Similarly, for a fixed $j$, we see that $f_ig_j\rightarrow fg_j$ on compact subset of $\Sigma$, so that
  \begin{equation}\label{fg equal fg H class2}
    (fg_j)(T)u=\lim_{i\rightarrow\infty}(f_ig_j)(T)u
  \end{equation}
  for any $u\in \mathcal{H}$. Finally, Proposition \ref{fg equal fg} and \eqref{fg equal fg H class1}, \eqref{fg equal fg H class2} give
  \begin{equation*}
    (fg)(T)u=\lim_{j\rightarrow \infty}\left(\lim_{i\rightarrow\infty}(f_ig_j)(T)u\right)=
    \lim_{j\rightarrow \infty}\left(\lim_{i\rightarrow\infty}\left(f_i(T)g_j(T)u\right)\right)
  \end{equation*}

  \begin{equation*}
    =\lim_{j\rightarrow \infty}\left(f(T)g_j(T)u\right)=f(T)\lim_{j\rightarrow \infty}\left(g_j(T)u\right)=f(T)g(T)u
  \end{equation*}
  for each $u\in \mathcal{H}$.
\end{proof}

Next we show the convergence lemma for the $H^{\infty}(\Sigma)$ functional calculus.
\begin{prop}\label{convergent lemma3}
Let $\{f_n\}_{n=1}^{\infty}\subset H^{\infty}(\Sigma)$ be uniformly bounded sequence. Assume $f\in H^{\infty}(\Sigma)$ and $f_n\rightarrow f$ uniformly on compact subsets of $\Sigma$. Then $f_n(T)u\rightarrow f(T)u$ for any $u\in \mathcal{H}$.
\end{prop}

\begin{proof}
  Fix $u\in \mathcal{H}$. By Proposition \ref{convergent lemma2}, there exist a sequence $\{m_n\}_{n=1}^{\infty}\subset \mathrm{N}$ such that $m_n>n$ and
  \begin{equation}\label{convergent lemma31}
    \left\|\frac{im_n}{im_n-T}f_n(T)u-f_n(T)u\right\|\rightarrow 0
  \end{equation}
  as $n\rightarrow \infty$. On the other hand, the sequence $\left\{\frac{im_n}{im_n-z}f_n(z)\right\}_{n=1}^{\infty}\subset\Theta(\Sigma)$ is uniformly bounded and converges to $f$ on compact subsets of $\Sigma$. Therefore
  \begin{equation}\label{convergent lemma32}
    \left\|\frac{im_n}{im_n-T}f_n(T)u-f(T)u\right\|\rightarrow 0
  \end{equation}
  as $n\rightarrow \infty$. The triangle inequality and \eqref{convergent lemma31}, \eqref{convergent lemma32} imply that
  \begin{equation*}
    \|f_n(T)u-f(T)u\|\rightarrow 0.
  \end{equation*}
\end{proof}

\subsection{Important examples of the functional calculus} We conclude this section by considering several important examples.

Let us define the following functions on $\Sigma$
\begin{equation*}
  \pi_{\pm}(z)=
  \begin{cases}
    1, & \mbox{if } z\in \Sigma^{\pm} \\
    0, & \mbox{if } z\in\Sigma\setminus \Sigma^{\pm}
  \end{cases},
  \quad
  \pi_{0}(z)=
  \begin{cases}
    1, & \mbox{if } z\in \Sigma^{0} \\
    0, & \mbox{if } z\in\Sigma\setminus \Sigma^{0}
  \end{cases}
\end{equation*}
and the corresponding operators $\Pi_{\pm}:=\pi_{\pm}(T)$, $\Pi_{0}:=\pi_{0}(T)$.

\begin{prop}\label{projections}
  The operators $\Pi_{\pm}$ and $\Pi_0$ are bounded complementary projections.
\end{prop}
\begin{proof}
  By Proposition \ref{fg equal fg H class}, we see that
  \begin{equation*}
    \Pi_{\pm}\Pi_{\pm}u=\pi_{\pm}(T)\pi_{\pm}(T)u=(\pi_{\pm}\pi_{\pm})(T)u=\pi_{\pm}(T)u=\Pi_{\pm}u
  \end{equation*}
  for any $u\in \mathcal{H}$. Similarly, we obtain
  \begin{equation*}
    \Pi_{0}\Pi_{0}u=\Pi_{0}u,
    \quad
    \Pi_0\Pi_{\pm}u=0,
    \quad
    \Pi_{\pm}\Pi_{\mp}u=0.
  \end{equation*}
  Since $(\pi_{-}+\pi_0+\pi_{+})(z)=1$ for $z\in \Sigma$, Propositions \ref{convergent lemma2} and \ref{fg equal fg H class} give
  $$
  \Pi_{-}u+\Pi_0u+\Pi_{+}u=u
  $$
  for any $u\in \mathcal{H}$.
\end{proof}

According to the above proposition, we have a topological splitting
\begin{equation*}
  \mathcal{H}=\mathbf{R}(\Pi_-)\oplus \mathbf{R}(\Pi_0)\oplus  \mathbf{R}(\Pi_+).
\end{equation*}

For given $u\in\mathbf{R}(\Pi_0)\oplus \mathbf{R}(\Pi_{\pm})$, we define
\begin{equation*}
  u_t:=\left(e^{-tT}\right)u
\end{equation*}
for $\pm t>0$, where $e^{-tT}$ is the operator obtained from the function
\begin{equation*}
  h_{t}(z)=
  \begin{cases}
    e^{-tz}, & \mbox{if } z\in\Sigma^0\cup \Sigma^{\pm}, \\
    0, & \mbox{if } z\in\Sigma\setminus \left(\Sigma^0\cup\Sigma^{\pm}\right),
  \end{cases}
\end{equation*}
by the functional calculus.

\begin{prop}\label{ode}
  Let $u\in \mathbf{R}(\Pi_0)\oplus\mathbf{R}(\Pi_{\pm})$. Then in $\mathcal{H}$ we have
  \begin{equation}\label{ode1}
    \partial_tu_t+Tu_t=0
  \end{equation}
  for $\pm t>0$ and $u_t\rightarrow u$ as $t\rightarrow 0$.
\end{prop}

\begin{proof}
  Let us fix $u\in\mathbf{R}(\Pi_0)\oplus \mathbf{R}(\Pi_{\pm})$. Note that $\partial_th_t(z)\in \Theta(\Sigma)$ and
  \begin{equation*}
    \frac{h_{t+\delta}(z)-h_t(z)}{\delta}\rightarrow \partial_th_t(z)
  \end{equation*}
  uniformly on compact subsets of $\Sigma$ as $\delta\rightarrow 0$. Therefore Proposition \ref{convergent lemma2} yields
  \begin{equation*}
    \partial_th_t(T)u=\left(\partial_th_t\right)(T)u=-Th_t(T)u.
  \end{equation*}
  This implies \eqref{ode1}.

  Next, for any compact subset of $\Sigma$, we have the uniform convergence of $h_t\in \Theta(\Sigma)$ to $\pi_0+\pi_+$ as $t\rightarrow 0$. Therefore Proposition \ref{convergent lemma2} gives
  \begin{equation*}
    \lim_{t\rightarrow 0}u_t=\lim_{t\rightarrow 0}h_t(T)u=(\Pi_0+\Pi_+)u=u.
  \end{equation*}

  The statement for $u\in\mathbf{R}(\Pi_{0})\oplus\mathbf{R}(\Pi_{-})$ follows from similar arguments.
\end{proof}


\section{Application to waveguide propagation}\label{Applications}
In this section, we return to the Helmholtz equation and Maxwell's system of equations and use our new functional calculus for the operator $T:=\left.DB\right|_{\mathbf{R}(D)}$ to investigate acoustic and electromagnetic waves along the waveguide. More precisely, in Theorems \ref{app1} and \ref{app2} we prove that all polynomially bounded time-harmonic waves in the semi- or bi-infinite waveguide have
representation in $\mathbf{R}(\Pi_0)$ or $\mathbf{R}(\Pi_0)\oplus\mathbf{R}(\Pi_+)$ respectively.

\subsection{The bi-infinite waveguide}
We start by considering the bi-infinite waveguide, that is we consider the ordinary differential equation
\begin{equation}\label{app11}
  (\partial_t+T)f=0,
  \qquad
  (t,x)\in \mathrm{R}\times\Omega.
\end{equation}

\begin{thm}\label{app1}
  $\mathbf{(A):}$ Let $f_0\in \mathbf{R}(\Pi^0)$ and
  \begin{equation*}
    h_t(z)=\begin{cases}
           e^{-tz}, & \mbox{if } z\in \Sigma^0 \\
           0, & \mbox{if } z\in \Sigma\setminus\Sigma^0.
         \end{cases}
  \end{equation*}
  Then $f_t:=h_t(T)f_0\in C(\mathrm{R};\mathbf{R}(\Pi_0))$ solves equation \eqref{app11} and for any nonnegative integer $j$ there exists a constant $C=C(j)>0$, which is independent of the choice of $f_0$, such that
  \begin{equation}\label{th1 polynomial}
    \|\partial_t^j f_t\|+\|T^jf_t\|<C(1+|t|^l)\|f_0\|
  \end{equation}
  with $l=\sup_{i}m_i-1$, where $m_i$ is the index of $\lambda_i^0$ for $i=1,...,N$.

  $\mathbf{(B):}$ Conversely, let $f_t\in C(\mathrm{R}; \mathcal{H})$ such that $f_t\in \mathbf{D}(T)$ for all $t\in \mathrm{R}$. Assume that $f_t$ solves equation \eqref{app11} and satisfies
  \begin{equation}\label{app14}
    \|f_t\|<C e^{\varepsilon|t|}
  \end{equation}
  for all $t\in\mathrm{R}$ and some $t$-independent constants $C>0$ and $\varepsilon\in (0,a)$. Then $f_0\in \mathbf{R}(\Pi^0)$ and for any $t\in \mathrm{R}$,
  \begin{equation*}
    f_t=h_t(T)f_0.
  \end{equation*}
\end{thm}

\begin{proof}
  $\mathbf{(A):}$ Note that $h_t(z)\in \Theta(\Sigma)$ for any $t\in \mathrm{R}$. Therefore Theorems \ref{convergent lemma1} and \ref{functional calculus} imply
  \begin{equation*}
    \|h_t(T)-h_{t+\delta}(T)\|\leq C\|e^{-tz}-e^{-(t+\delta)z}\|_{L^{\infty}(\Sigma^0)}\rightarrow 0
  \end{equation*}
  as $\delta\rightarrow 0$, so that $f_t\in C(\mathbf{R};\mathcal{H})$. By Proposition \ref{ode}, $f_t$ solves the equation \eqref{app11}.
  The boundedness of $\left.T\right|_{\mathbf{R}(\Pi_0)}$ and Proposition \ref{algebraic multiplicity} imply
  \begin{equation*}
    \|\partial_t^j f_t\|+\|T^jf_t\|\leq C\sum_{i=1}^{N}\max_{0\leq k\leq m_i-1}|h^{(k)}(\lambda^0_i)|\|f_0\|,
  \end{equation*}
  and thereby derive \eqref{th1 polynomial}.

  $\mathbf{(B):}$ Let us set
  \begin{equation*}
    g^+_t(z)=\begin{cases}
             e^{-tz}, & \mbox{if } z\in\Sigma^+, \\
             0, & \mbox{if } z\in\Sigma\setminus\Sigma^+
           \end{cases}
  \end{equation*}
  for $t>0$. By assumption, $f_t$ solves \eqref{app11}. Therefore, for $t_0\in \mathrm{R}$ and $t<t_0$, we obtain
  \begin{equation*}
    \partial_t\left(g^+_{t_0-t}(T)\Pi_+f_t\right)=g^+_{t_0-t}(T)\left(\partial_t+T\right)\Pi_+f_t=
    g^+_{t_0-t}(T)\Pi_+\left(\partial_t+T\right)f_t=0.
  \end{equation*}
  Integrating over $(P,t_0)$ for some $P<t_0$, gives
  \begin{equation*}
    \Pi_+f_{t_0}-g^+_{t_0-P}(T)\Pi^+f_P=0.
  \end{equation*}
  By Theorem \ref{functional calculus} and estimate \eqref{app14}, we obtain
  \begin{equation*}
  \|g^+_{t_0-P}(T)\Pi_+f_P\|\leq C\sup_{z\in \Sigma^+}\left|e^{-(t_0-P)z}\right|e^{\varepsilon|P|}\leq Ce^{-(t_0-P)a}e^{\varepsilon|P|}.
  \end{equation*}
  Letting $P\rightarrow-\infty$, we conclude that $\Pi_+f_{t_0}=0$ for $t_0\in \mathrm{R}$.

  Similarly, let
  \begin{equation*}
    g^-_t(z)=\begin{cases}
             e^{tz}, & \mbox{if } z\in\Sigma^-, \\
             0, & \mbox{if } z\in\Sigma\setminus\Sigma^-
           \end{cases}
  \end{equation*}
  for $t>0$. Then, for $t_0\in \mathrm{R}$ and $t>t_0$, we derive
  \begin{equation*}
    \partial_t\left(g^-_{t-t_0}(T)\Pi_-f_t\right)=0.
  \end{equation*}
  By integrating over $(t_0,P)$ and letting $P\rightarrow+\infty$, we conclude $\Pi_-f_{t_0}=0$ for $t_0\in \mathrm{R}$, and hence $f_0\in \mathbf{R}(\Pi^0)$. Then the first part of this theorem implies that $\widetilde{f_t}=h_t(T)f_0$ solves equation \eqref{app11}, and hence
  \begin{equation*}
    \partial_t(h_{s-t}(T)(f_t-\widetilde{f_t}))=0
  \end{equation*}
  for $t<s$. By integrating over $(P,s)$ and letting $P\rightarrow0$, one can prove $f_s=\widetilde{f_s}$, so that $f_t=h_t(T)f_0$.
\end{proof}

\subsection{The semi-infinite waveguide}
Next to obtain a similar result for the semi-infinite waveguide we consider the ordinary differential equation
\begin{equation}\label{app21}
  (\partial_t+T)f=0,
  \qquad
  (t,x)\in \mathrm{R}^+\times\Omega,
\end{equation}
where $\mathrm{R}^+:=(0,+\infty)$.

\begin{thm}\label{app2}
  $\mathbf{(A):}$ Let $f_0\in \mathbf{R}(\Pi_0)\oplus \mathbf{R}(\Pi_+)$ and
  \begin{equation*}
    h_t(z)=\begin{cases}
           e^{-tz}, & \mbox{if } z\in \Sigma^0\cup\Sigma^+, \\
           0, & \mbox{if } z\in \Sigma^-
         \end{cases}
  \end{equation*}
  for $t>0$. Then $f_t:=h_t(T)f_0\in C(\mathrm{R}^+; \mathbf{R}(\Pi_0)\oplus \mathbf{R}(\Pi_+))$ solves equation \eqref{app21} and for any nonnegative integer $j$ there exists a constant $C=C(j)>0$, which is independent of the choice of $f_0$, such that
  \begin{equation}\label{app22}
    \|\partial_t^j f_t\|+\|T^jf_t\|<C(t^l+t^{-j})\|f_0\|
  \end{equation}
  with $l=\sup_{i}m_i-1$, where $m_i$ is the index of $\lambda_i^0$ for $i=1,...,N$. Moreover, $\lim_{t\rightarrow0}f_t=f_0$ in $\mathcal{H}$.

  $\mathbf{(B):}$ Conversely, let $f_t\in C(\mathrm{R}^+; \mathcal{H})$ such that $f_t\in \mathbf{D}(T)$ for all $t\in \mathrm{R}^+$. Assume that $f_t$ solves \eqref{app21} and satisfies
  \begin{equation}\label{app24}
    \|f_t\|<C e^{\varepsilon|t|}
  \end{equation}
  for all $t\in\mathrm{R}^+$ and some $t$-independent constants $C>0$ and $\varepsilon\in(0,a)$. Then, there exists $f_0\in \mathbf{R}(\Pi_0)\oplus\mathbf{R}(\Pi_+)$ such that
  \begin{equation*}
    f_t=h_t(T)f_0
  \end{equation*}
  for $t\in \mathrm{R}^+$. Moreover, $f_0\in \mathbf{R}(\Pi_+)$ if and only if $\|f_t\|\rightarrow 0$ as $t\rightarrow\infty$.
  \end{thm}

\begin{proof}
  $\mathbf{(A):}$ By Proposition \ref{ode}, $f_t$ solves equation \eqref{app21}. From Theorem \ref{app1}, we see that
  \begin{equation}\label{app25}
    \|\partial_t^jh_t(T)\Pi_0f_0\|+\|T^jh_t(T)\Pi_0f_0\|\leq C(1+|t|^l)\|\Pi_0f_0\|.
  \end{equation}
  Theorem \ref{functional calculus} implies now that
  \begin{equation*}
    \|\partial_t^jh_t(T)\Pi_+f_0\|+\|T^jh_t(T)\Pi_+f_0\|\leq \sup_{z\in \Sigma^+}\left|(1+z^j)e^{-tz}\right|\|\Pi_+f_0\|.
  \end{equation*}
  This gives
  \begin{equation}\label{app2.6.1}
    \|\partial_t^jh_t(T)\Pi_+f_0\|+\|T^jh_t(T)\Pi_+f_0\|\leq C t^{-j}\|\Pi_+f_0\|
  \end{equation}
  as $t\rightarrow 0$, and
  \begin{equation}\label{app2.6.2}
    \|\partial_t^jh_t(T)\Pi_+f_0\|+\|T^jh_t(T)\Pi_+f_0\|\leq C e^{-ta}\|\Pi_+f_0\|
  \end{equation}
  as $t\rightarrow \infty$. Combining \eqref{app25}-\eqref{app2.6.2}, we obtain the estimate \eqref{app22}.

  Since $h_t\rightarrow\pi_0+\pi_+$ uniformly on compact subsets of $\Sigma$, Proposition \ref{convergent lemma2} implies
  \begin{equation*}
    \lim_{t\rightarrow 0}h_t(T)f_0=(\pi_0(T)+\pi_+(T))f_0=(\Pi_0+\Pi_+)f_0=f_0.
  \end{equation*}

  $\mathbf{(B):}$ Let us set
  \begin{equation*}
    g_t(z)=\begin{cases}
             e^{tz}, & \mbox{if } z\in\Sigma^-, \\
             0, & \mbox{if } z\in\Sigma\setminus\Sigma^-
           \end{cases}
  \end{equation*}
  for $t>0$. By our assumption, $f_t$ solves \eqref{app21}. Therefore, for $t>s>0$, we obtain
  \begin{equation*}
    \partial_t\left(g_{t-s}(T)\Pi_-f_t\right)=g_{t-s}(T)(\partial_t+T)\Pi_-f_t=g_{t-s}(T)\Pi_-(\partial_t+T)f_t=0.
  \end{equation*}
  Integrating over $(s,P)$ for some $P>s$, gives
  \begin{equation*}
    g_{P-s}(T)\Pi_-f_P-\Pi_-f_{s}=0.
  \end{equation*}
  Theorem \ref{functional calculus} and estimate \eqref{app24} imply now that
  \begin{equation*}
    \|g_{P-s}(T)\Pi_-f_P\|\leq Ce^{-aP}e^{\varepsilon P}.
  \end{equation*}
  Letting $P\rightarrow\infty$, we get $\Pi_-f_s=0$, so that $f_s\in \mathbf{R}(\Pi_0)\oplus \mathbf{R}(\Pi_+)$ for all $s\in \mathrm{R}^+$.

  Fix $s>0$. The first part of this theorem implies that $f_{s+t}-h_t(T)f_s$ solves \eqref{app21}, and hence
  \begin{equation*}
    \partial_t(h_{r-t}(T)(f_{s+t}-h_t(T)f_s))=0
  \end{equation*}
  for $0<t<r$. Let $\varepsilon\in(0,r)$, then integration over $(P,r-\varepsilon)$ gives
  \begin{equation*}
    h_{r-(r-\varepsilon)}(T)(f_{s+(r-\varepsilon)}-h_{r-\varepsilon}(T)f_s)-h_{r-P}(T)(f_{s+P}-h_P(T)f_s)=0.
  \end{equation*}
  Letting $P$, $\varepsilon\rightarrow0$, we obtain $f_{s+r}-h_r(T)f_s=0$ for $r>0$, or equivalently
  \begin{equation}\label{app27}
    f_t=h_{t-s}(T)f_{s}
  \end{equation}
  for $0<s<t$.

  Since $f_t$ is uniformly bounded as $t\rightarrow 0$, one can find a decreasing sequence $\{s_k\}_{k=1}^{\infty}\subset \mathrm{R}^+$ such that $s_k\rightarrow 0$ and $f_{s_{k}}\rightarrow f_0$ weakly in $\mathcal{H}$. Let $\phi$ be a test function. Then, due to \eqref{app27},
  \begin{equation*}
    \left(f_t, \phi\right)=\left(h_{t-s_k}(T)f_{s_k},\phi\right)=\left(f_{s_k},h_{t-s_k}(T)^*\phi\right)
  \end{equation*}

  \begin{equation*}
    =\left(f_{s_k},h_{t-s_k}(T)^*\phi-h_{t}(T)^*\phi\right)+\left(f_{s_k},h_{t}(T)^*\phi\right)
  \end{equation*}
  for $t>s_k$. Therefore
  \begin{equation*}
    \left|\left(f_t, \phi\right)-\left(f_{s_k},h_{t}(T)^*\phi\right)\right|\leq \|f_{s_k}\|\left\|h_{t-s_k}(T)^*\phi-h_{t}(T)^*\phi\right\|
  \end{equation*}
  and letting $k\rightarrow \infty$, we obtain
  \begin{equation*}
    \left|\left(f_t, \phi\right)-\left(f_{0},h_{t}(T)^*\phi\right)\right|\leq 0.
  \end{equation*}
  Hence $f_t=h_{t}(T)f_0$. Since $h_t\rightarrow \pi_0+\pi_+$ uniformly on compact subsets of $\Sigma$, we conclude $f_t\rightarrow f_0$ strongly in $\mathcal{H}$.

  Finally, if $f_0\in \mathbf{R}(\Pi_+)$, then
  \begin{equation*}
    \|f_t\|\leq  Ce^{-at}\|f_0\|
  \end{equation*}
  for $t>0$. Hence $\|f_t\|\rightarrow 0$ as $t\rightarrow \infty$.

  Conversely, assume $\|f_t\|\rightarrow 0$ as $t\rightarrow \infty$. Then $\|h_t(T)\Pi_0f_0\|\rightarrow 0$ as $t\rightarrow \infty$, and therefore
  \begin{equation*}
    \left\|\sum_{i=1}^{N}\sum_{k=0}^{m_i-1}\frac{(-t)^ke^{-t\lambda_i^0}}{k!}(\lambda_i^0-T)^k\Pi_i\Pi_0f_0\right\|\rightarrow 0
  \end{equation*}
  as $t\rightarrow \infty$. Since $\mathbf{R}(\Pi_0)=\bigoplus_{j=1}^{N}\mathbf{R}(\Pi_j)$ and $\lambda^0_i$ is purely imaginary, we obtain
  \begin{equation}\label{app29}
    \left\|\sum_{k=0}^{m_i-1}\frac{(-t)^k}{k!}(\lambda_i^0-T)^k\Pi_if_0\right\|\rightarrow 0
  \end{equation}
  as $t\rightarrow\infty$ for $i=1,...,N$. Let us define
  \begin{equation*}
  l_i:=\sup\{k=1,...,m_i-1:\; (\lambda_i^0-T)^k\Pi_if_0\neq0\}.
  \end{equation*}
  If $l_i>0$, the identity
  \begin{equation*}
    t^{l_i}=(t-1)(t^{l_i-1}+t^{l_i-2}+...+1)+1
  \end{equation*}
  implies
  \begin{equation*}
    \left\|\sum_{k=0}^{m_i-1}\frac{(-t)^k}{k!}(\lambda_i^0-T)^k\Pi_if_0\right\|\geq
    \frac{1}{2}\left\|\frac{(-t)^{l_i}}{l_i!}(\lambda_i^0-T)^{l_i}\Pi_if_0\right\|
  \end{equation*}
  for sufficiently large $t>0$. By \eqref{app29} the left hand side tends to $0$ as $t\rightarrow\infty$, while the right hand side tends to $\infty$. This contradiction shows that $l_i=0$ for $i=1,...,N$. Hence \eqref{app29} implies $\Pi_if_0=0$ for $i=1,...,N$. Therefore $f_0\in \mathbf{R}(\Pi_+)$.
\end{proof}


\subsection*{Acknowledgment}
The authors are greatly indebted to Grigori Rozenblum (The University of Gothenburg, Chalmers University of Technology), Lashi Bandara (The University of Gothenburg, Chalmers University of Technology) and Julie Rowlett (The University of Gothenburg, Chalmers University of Technology) for helpful comments and suggestions.

\end{document}